\documentclass[10pt]{article}
\usepackage{amsmath,amsthm}
\usepackage{appendix}
\usepackage{amssymb}
\usepackage{graphicx}
\usepackage{epsfig}
\usepackage{color}
\usepackage{subfig}
\usepackage{newfloat}
\usepackage{algorithm}
\usepackage{algorithmic}
\usepackage{caption}
\usepackage{booktabs}
\usepackage[detect-all]{siunitx}
\usepackage{etoolbox}
\usepackage{multirow}
\usepackage{mathtools}
\usepackage{tabulary}
\usepackage{tabularx}
\usepackage{array}
\usepackage{enumerate} 
\usepackage{comment} 
\usepackage{url}
\usepackage{hyperref}
\usepackage[a4paper,tmargin=1.5in,bmargin=1.5in]{geometry}
\usepackage[normalem]{ulem}

\newtheorem{theo}{Theorem}[section] 
\newtheorem{defi}[theo]{Definition}

\newtheorem{lemm}[theo]{Lemma}
\newtheorem{rema}[theo]{Remark}
\newtheorem{coro}[theo]{Corollary}

\newtheorem{assumption}[theo]{Assumption}

\newcommand{\beq}{\begin{equation}}
\newcommand{\eeq}{\end{equation}}
\newcommand{\beqs}{\begin{equation*}}
\newcommand{\eeqs}{\end{equation*}}

\newcommand{\C}{\mathbb C}
\newcommand{\R}{\mathbb R}

\newcommand{\bigO}{{\cal O}}

\newcommand{\imagunit}{{\bf i}}

\newcommand{\hinfsym}{\mathcal{H}}
\newcommand{\linfsym}{\mathcal{L}}

\newcommand{\Hinf}{{\hinfsym_\infty}}
\newcommand{\Linf}{{\linfsym_\infty}}

\newcommand{\Cmn}[2]{\C^{{#1}\times{#2}}}  

\renewcommand\Re{\operatorname{Re}}
\renewcommand\Im{\operatorname{Im}}

\newcommand{\Real}[1]{\Re{\left({#1}\right)}}

\DeclareMathOperator*{\argmax}{arg\,max}

\DeclareMathOperator*{\Arg}{Arg}

\DeclareFloatingEnvironment[placement=htb]{algfloat}

\newcounter{subroutine}
\makeatletter
\newenvironment{subroutine}[1][htb]{%
  \let\c@algorithm\c@subroutine
  \renewcommand{\ALG@name}{Subroutine}
  \begin{algorithm}[#1]%
  }{\end{algorithm}
}
\makeatother

\def\emach{\epsilon_\mathrm{mach}}
\def\eps{\varepsilon}



\newcommand\fnsub[3][]{
 	\ifstrempty{#1}{	
		#2_{#3}
	}{
		#2_{#3\resizebox{0.7\width}{0.8\height}{=}#1}
	}
}
\newcommand\ntfx[1][]{\fnsub[#1]{g}{x}}
\newcommand\ntfy[1][]{\fnsub[#1]{g}{y}}
\newcommand\ntft[1][]{\fnsub[#1]{g}{\theta}}

\def\cs{\Omega}

\def\eigderivmat{H}

\def\abcde{(A,B,C,D,E)}

\newcommand{\aleps}{\alpha_\varepsilon}

\newcommand{\rhoeps}{\rho_\varepsilon}

\newcommand{\SVSeps}{\spec_\eps}

\newcommand{\spec}{\sigma} 

\newcommand{\norm}[1]{\left\lVert#1\right\rVert}

\newcommand{\norms}[1]{\norm{#1}_2}

\def\Mc{\mathcal{M}_{\gamma x}}
\def\Nc{\mathcal{N}}
\def\Md{\mathcal{S}_{\gamma r}}
\def\Nd{\mathcal{T}_{\gamma r}}

\def\MNpencont{(\Mc,\Nc)}
\def\MNpendisc{(\Md,\Nd)}

\newcommand{\tfqs}[1]{C\left(#1 E -A\right)^{-1}B + D}
\def\eitheta{e^{\imagunit \theta}}
\def\eithetaconj{e^{-\imagunit \theta}}
\def\xiy{x{+}\imagunit y}
\def\xiyconj{x{-}\imagunit y}
\def\iy{\imagunit y}

\newcommand{\mytilde}{\raise.17ex\hbox{$\scriptstyle\mathtt{\sim}$}}
\newcommand{\matlab}{{MATLAB}}

\newcommand{\removable}[1]{{\color{Magenta}#1}}
\renewcommand{\removable}[1]{}

\title{Extended and improved criss-cross algorithms\\for computing the spectral value set\\abscissa and radius}
\author{
Peter Benner$^*$
\and 
Tim Mitchell\thanks{
Max Planck Institute for Dynamics of Complex Technical Systems, Magdeburg, 39106 Germany \href{mailto:benner@mpi-magdeburg.mpg.de}{\texttt{benner@mpi-magdeburg.mpg.de}}, \href{mailto:mitchell@mpi-magdeburg.mpg.de}{\texttt{mitchell@mpi-magdeburg.mpg.de}}.}
}
\date{September 6th, 2019}
\begin{document} 
\maketitle

\begin{abstract}
In this paper, we extend the original criss-cross algorithms for computing 
the \mbox{$\eps$-}\nobreak\hspace{0pt}pseudospectral abscissa and radius to general spectral value sets.
By proposing new root-finding-based strategies for the horizontal/radial search subphases,
we significantly reduce the number of expensive Hamiltonian eigenvalue decompositions incurred,
which typically translates to meaningful speedups in overall computation times.
Furthermore, and partly necessitated by our root-finding approach, 
we develop a new way of handling the singular pencils or problematic interior searches that can arise 
when computing the $\eps$-spectral value set radius.
Compared to would-be direct extensions of the original algorithms, that is, without our additional modifications, 
our improved criss-cross algorithms are not only noticeably faster but also more robust and numerically accurate, for both spectral value set and pseudospectral problems.
\end{abstract}

\section{Introduction}
\label{sec:intro}

Consider the continuous-time linear dynamical system 
\begin{subequations}
\label{eq:lti_cont}
\begin{align}
\label{eq:lti_ABcont}
E\dot{x} &=  Ax + Bu, \\
\label{eq:lti_CDcont}
y &=  Cx + Du,
\end{align}
\end{subequations}
where $A \in \Cmn{n}{n}$, $B \in \Cmn{n}{m}$, $C \in \Cmn{p}{n}$, $D \in \Cmn{p}{m}$, and $E \in \Cmn{n}{n}$
is assumed to be invertible.
Using output feedback $u = \Delta y$, where $\Delta \in \Cmn{m}{p}$, 
so that input $u$ varies linearly with respect to output $y$, 
\eqref{eq:lti_ABcont} can be rewritten as $E\dot{x} = Ax + B\Delta y$ and \eqref{eq:lti_CDcont} as  $y = (I - D\Delta)^{-1}Cx$,
assuming $(I - D\Delta)$ is invertible.
Thus, the input-output system \eqref{eq:lti_cont} is equivalent to
\beq
	\label{eq:EM}
	E\dot{x} = M(\Delta)x,
\eeq
where 
\beq
	\label{eq:MDelta}
	M(\Delta) \coloneqq A + B\Delta(I - D\Delta)^{-1}C
\eeq
is called the \emph{perturbed system matrix}.  
As a consequence, the dynamical properties of \eqref{eq:lti_cont},
which arise in many engineering applications,
can be studied by examining the generalized eigenvalue problem of the matrix pencil $(M(\Delta),E) = \lambda E - M(\Delta)$.  

For the special case of $B=C=E=I_n$, where $I_n$ is the $n \times n$ identity matrix, and $D=0$, \eqref{eq:EM} simply reduces to 
\beq
	\label{eq:ADelta}
	\dot x = (A + \Delta) x.
\eeq
Considering $\Delta=0$, the ordinary differential equation $\dot x = Ax$ is asymptotically stable if its 
\emph{spectral abscissa}, the maximal real part attained by the eigenvalues of matrix $A$, is strictly negative: $\alpha(A) < 0$.
However, the spectrum only provides a limited perspective with respect to the dynamics of the system.
If matrices close to an asymptotically stable matrix $A$ have eigenvalues in the right half plane, 
then the solution of $\dot x = Ax$ may still have large transient behavior before converging.  
Furthermore, in applications, where $A$ models some physical process or mechanism, 
the theoretical asymptotic stability of $A$ may not be predictive of reality, 
particularly if small perturbations of the model $A$ can result in unstable systems.
Hence, there has been great interest to also consider the dynamical properties of \eqref{eq:ADelta},
which is characterized by \emph{pseudospectra} \cite{TreE05}: the set of eigenvalues of $A$ under general perturbation, typically limited 
by placing an upper bound on the spectral norm of $\Delta$.
For a given $\eps \ge 0$, the \emph{$\eps$-pseudospectral abscissa}:
\[
	\aleps(A) \coloneqq \max \{ \Re \lambda : \lambda \in \spec(A+ \Delta), \norms{\Delta} \le \eps \},
\]
where $\spec(\cdot)$ denotes the spectrum, provides a measure of \emph{robust stability}: 
if $\aleps(A) < 0$, then $A + \Delta$ is stable for any perturbation such that $\norms{\Delta} \le \eps$.
The norm of the smallest destabilizing perturbation, i.e., the value of $\eps$ that yields $\aleps(A) = 0$,
is called the \emph{distance to instability}, introduced by \cite{Van85a}.  
Beyond robust stability measures, pseudospectra also provide information about the transient behaviors of dynamical systems [TE05, Chs. 14-19].  For example, [TTRD93] proposed pseudospectra as a tool for analyzing how laminar flows transition to turbulence, by looking not just at spectra but pseudospectra of (stable) linearizations of the nonlinear problem.

Computationally, numerous techniques for plotting the boundaries of pseudospectra are discussed in \cite{Tre99,WriT01},
while a ``criss-cross" algorithm for computing the $\eps$-pseudospectral abscissa to high precision, 
with a local quadratic rate of convergence, was proposed in \cite{BurLO03}.  
The criss-cross algorithm performs a sequence of alternating vertical and horizontal searches
to find relevant boundary points of the $\eps$-pseudospectrum along the respective search lines, 
which converge to a globally rightmost point of the $\eps$-pseudospectrum; 
these vertical and horizontal searches are accomplished by computing eigenvalues of associated Hamiltonian matrices or matrix pencils.  
In fact, the techniques used in the criss-cross algorithm build upon those developed 
for the first algorithm for computing the distance to instability \cite{Bye88}.
Relevant for discrete-time systems $x_{k+1} = Ax_k$, the criss-cross algorithm has also been adapted to compute the corresponding \emph{$\eps$-pseudospectral radius}:
\[
	\rhoeps(A) \coloneqq \max \{ | \lambda | : \lambda \in \spec(A+ \Delta), \norms{\Delta} \le \eps \},
\]
by using circular and radial searches instead of vertical and horizontal ones \cite{MenO05}.
Of course, when $\eps=0$, $\rhoeps(A) = \rho(A)$, the \emph{spectral radius} of $A$.

For the more general setting of \eqref{eq:lti_cont}, 
the analogue of the $\eps$-pseudospectrum is an \emph{$\eps$-spectral value set}
while the analogue of the distance to instability is the \emph{complex stability radius} 
(perhaps better known by its reciprocal value, the $\Hinf$ norm).
Spectral value sets are distinctly different from pseudospectra of generalized eigenvalue problems $\lambda E - A$,
where both $A$ and $E$ could be considered under general perturbation.
In spectral value sets, \eqref{eq:lti_cont} only permits \emph{structured} perturbations of the form $B\Delta(I - D\Delta)^{-1}C$
to operator $A$, while $E$ remains unperturbed.
Fixed matrices $A$, $B$, $C$, $D$, and $E$ represent the certainties of the model 
while $\Delta$ represents the uncertainties in the feedback loop.  
To identify dynamical properties of \eqref{eq:lti_cont},
it is natural to consider the worst outcome possible over the set of uncertainties. 
The complex stability radius encodes precisely that: the norm of the smallest matrix $\Delta$ 
such that $B\Delta(I - D\Delta)^{-1}C$ destabilizes \eqref{eq:lti_cont}, assuming for now that $(A,E)$ is stable itself.

Algorithms for computing the complex stability radius (or the $\Hinf$ norm) 
of general systems with input and output \eqref{eq:lti_cont} 
also generally rely on extensions of the level set techniques developed by \cite{Bye88} for 
computing the distance to instability.
Like the pseudospectral abscissa and radius algorithms, 
these too require $\bigO(n^3)$ amount of work and $\bigO(n^2)$ memory per iteration 
so there has been much recent interest in developing alternative scalable approximation techniques.
Spectral value sets have been a useful tool in this endeavor
(see \cite{GugGO13,BenV14,MitO16}), 
even though exact methods have not made use of them (at least not directly).
The key component has been the introduction of efficient iterations 
for approximating the \emph{$\eps$-spectral value set abscissa}, which was first done
for approximating the $\eps$-pseudospectral abscissa (and radius) in \cite{GugO11}.

In this work, we extend the pseudospectral methods of \cite{BurLO03,MenO05}
to computing the spectral value set and radius,
thus providing dense and exact analogues to the above scalable approximation techniques.
We also propose significant modifications and improvements to these methods.
The core idea is one we simultaneously exploited in our work to 
accelerate the computation of the $\Hinf$ norm \cite{BenM18a}:
replace large Hamiltonian eigenvalue computations  
with much cheaper evaluations of the norm of the transfer function wherever possible.
However, while \cite{BenM18a} uses a rather straightforward 
application of smooth optimization techniques to take larger (and thus fewer) 
steps before converging to the $\Hinf$ norm,
our work here involves several important differences and additional complexities.
First, we replace the globally-optimal horizontal/radial searches in the original 
algorithms with much cheaper but possibly only locally-optimal root-finding-based searches 
(using the norm of the transfer function); 
consequently, our new algorithms could conceivably incur more iterations than the original methods,
even though they are often significantly faster overall.
Second, our new approach also affords a new strategy to intelligently order the horizontal/radial searches so that relatively few are actually solved per iteration and those that are solved are all warm started by increasingly better initializations.
Third, as the original pseudospectral radius algorithm requires globally-optimal radial searches
to ensure it does not stagnate, 
we additionally propose a new technique for overcoming the problematic singular pencils and interior searches
that may arise, one that is both compatible with our new locally-optimal radial searches and
that should also be more robust in practice.
While our modifications only affect the constant factors in terms of efficiency,
the resulting speedups are nevertheless typically meaningful.
For example, in robust control applications, 
the spectral value set (or pseudospectral) abscissa/radius can appear 
as part of a nonsmooth optimization design task and will thus be typically
evaluated thousands or even millions of times during optimization.
Finally, by no longer computing purely imaginary eigenvalues of Hamiltonian eigenvalue problems for the horizontal/radial searches,
our new methods also avoid the accompanying rounding errors of such computations;
as a result, our improved methods are more reliable and accurate in practice.

The paper is organized as follows. 
Prerequisite definitions and theory are given in \S\ref{sec:svs}.
In \S\ref{sec:abs_old}, we directly extend the pseudospectral abscissa algorithm of \cite{BurLO03}
to the spectral value set abscissa and then present our corresponding improved method in \S\ref{sec:abs_new}.  We respectively do the same for the pseudospectral radius algorithm of \cite{MenO05} and the spectral value set radius in \S\ref{sec:rad_old} and \S\ref{sec:rad_new}, the latter of which includes our new way of handling singular pencils and interior searches.  Convergence results are given in \S\ref{sec:convergence}, while implementation 
details and numerical experiments are respectively provided in \S\ref{sec:code} and \S\ref{sec:numerical}.  
Concluding remarks are made in \S\ref{sec:wrapup}.

\section{Spectral value sets and the transfer function}
\label{sec:svs}
The following general concepts are used throughout the paper.

\begin{defi}
Given a nonempty closed set $\mathcal{D} \subset \C$, a point $\lambda \in \mathcal{D}$ is:
\begin{enumerate}
\item \emph{rightmost} if $\Re \lambda = \max \{ \Re z : z \in \mathcal{D} \}$
\item \emph{outermost} if $| \lambda | = \max \{ | z | : z \in \mathcal{D} \}$
\item \emph{isolated} if $\mathcal{D} \cap \mathcal{N} = \lambda$ for some neighborhood $\mathcal{N}$ of $\lambda$
\item \emph{interior} or \emph{strictly inside} if $\mathcal{N} \subset \mathcal{D}$ for some neighborhood $\mathcal{N}$ of $\lambda$.
\end{enumerate}
Furthermore, $\lambda$ is a \emph{locally rightmost or outermost} point of $\mathcal{D}$ 
if $\lambda$ is respectively a rightmost or outermost point
of $\mathcal{D} \cap \mathcal{N}$, for some neighborhood $\mathcal{N}$ of $\lambda$.
\end{defi}

\begin{defi}
\label{def:svs}
Let $\eps \ge 0$ be such that $\eps \|D\|_2 < 1$ and define the \emph{$\eps$-spectral value set}
\beq
	\SVSeps(A,B,C,D,E) = \bigcup \, \{ \spec(M(\Delta),E) : \Delta \in \Cmn{m}{p}, \norms{\Delta} \le \eps \}.
\eeq
\end{defi}

\begin{rema}
Note that we assume that $E$ is invertible, here and throughout the paper.  
If $E$ is singular but $(A,E)$ is still index 1,
then the system can be transformed into an equivalent representation without a singular $E$ matrix;
see \cite{morFreRM08} for details.  
\end{rema}

Now consider the \emph{transfer function} associated with input-output system \eqref{eq:lti_cont}:
\beq
	\label{eq:tf}
	G(\lambda) \coloneqq C(\lambda E - A)^{-1} + D \quad \text{for} \quad \lambda \in \C \backslash \spec(A,E). 
\eeq

As shown in \cite[\S5.2]{HinP05} for $E=I$, 
spectral value sets can be equivalently defined in terms of the norm of the transfer function,
instead of eigenvalues of $(M(\Delta),E)$.
This fundamental result easily extends to the case of generic $E$ matrices we consider here; 
e.g. the proof of \cite[Theorem 2.1]{GugGO13} readily generalizes by substituting 
all occurrences of $(\lambda I - A)$ with $(\lambda E - A)$.

\begin{theo}
\label{thm:eig_ntf}
Let $\eps \ge 0$ be such that $\eps \|D\|_2 < 1$ and $\| \Delta \|_2 \le \eps$ so that $I - D\Delta$ is invertible.
Then for $\lambda \not\in \spec(A,E)$ the following are equivalent:
\beq
	\|G(\lambda)\|_2 \ge \eps^{-1} \quad \text{and} \quad \lambda \in \spec(M(\Delta),E) \text{ for some } \Delta \text{ with } \|\Delta \|_2 \le \eps.
\eeq
\end{theo}

By Theorem~\ref{thm:eig_ntf}, the following corollary is immediate, 
providing an alternate spectral value set definition based on the norm of the transfer function.

\begin{coro}
\label{cor:svs_tf}
Let $\eps \ge 0$ be such that  $\eps \|D\|_2 < 1$.  Then 
\beq
	\SVSeps\abcde  = \spec(A,E) \, \bigcup \, \{ \lambda \in \C\backslash \spec(A,E) : \|G(\lambda)\|_2 \ge \eps^{-1}\}.
\eeq
\end{coro}

Note that the nonstrict inequalities in Definition~\ref{def:svs} and Theorem~\ref{thm:eig_ntf} imply
that the spectral value sets we consider are compact.
Furthermore, the boundary of $\SVSeps\abcde$ is characterized by the condition $\|G(\lambda)\|_2 = \eps^{-1}$
while for any matrix $\Delta$ such that $\lambda \in \spec(M(\Delta),E)$ is a boundary point,
 $\|\Delta\|_2=\eps$ must hold (though the reverse implication is not true).

\begin{lemm}
\label{lem:path}
Let $\eps > 0$ be such that $\eps\|D\|_2 < 1$ and 
let $\lambda$ be a non-isolated boundary point of an $\eps$-spectral value set,
with associated perturbation matrix $\Delta$, that is, $\lambda \in \spec(M(\Delta),E)$.
Then for one or more $\lambda_0 \in \spec(A,E)$, there exists a continuous path parameterized by $t \in [0,1]$
such that $\lambda(t)$ is an eigenvalue of $\spec(M(t\Delta),E)$ 
taking $\lambda(0) = \lambda_0$ to $\lambda(1) = \lambda$.
Furthermore, $\lambda(t)$ is only a boundary point at $t=1$.
\end{lemm}
\begin{proof}
By continuity of eigenvalues, the continuous path $\lambda(t)$ exists
and clearly, $\|t\Delta\|_2 < \|\Delta\|_2$ holds for $t \in [0,1)$.
As $\lambda$ is on the boundary, $\|\Delta\|_2 = \eps$ holds
but then the necessary condition for 
$\lambda(t)$ to be a boundary point is violated for all $t \in [0,1)$.
\end{proof}
 
\subsection{The spectral value set abscissa and radius}
\label{sec:svs_absrad}
The $\eps$-spectral value set abscissa, relevant for continuous-time systems \eqref{eq:lti_cont}, is formally defined as follows.

\begin{defi}
Let $\eps \ge 0$ be such that $\eps \|D\|_2 < 1$ and define the \emph{$\eps$-spectral value set abscissa}
\beq
	\aleps(A,B,C,D,E) \coloneqq \max \{ \Re \lambda : \lambda \in \SVSeps(A,B,C,D,E) \}.
\eeq
\end{defi}

Now consider the discrete-time linear dynamical system 
\begin{subequations}
\label{eq:lti_disc}
\begin{align}
Ex_{k+1} &=  Ax_k + Bu_k \\
y_k & =  Cx_k + Du_k,
\end{align}
\end{subequations}
where the matrices are defined as before in \eqref{eq:lti_cont}. 
For the case when $B=C=E=I_n$, and $D=0$, 
the simple ordinary difference equation $x_{k+1} = Ax_k$ is asymptotically stable 
if and only if its \emph{spectral radius},
the maximal modulus attained by the eigenvalues of $A$, 
is strictly less than one: $\rho(A) < 1$.
Thus, for discrete-time input-output systems of the form of \eqref{eq:lti_disc},
we generalize the $\eps$-pseudospectral radius as follows.

\begin{defi}
Let $\eps \ge 0$ be such that $\eps \|D\|_2 < 1$ and define the \emph{$\eps$-spectral value set radius}
\beq
	\rhoeps(A,B,C,D,E) \coloneqq \max \{ | \lambda | : \lambda \in \SVSeps(A,B,C,D,E) \}.
\eeq
\end{defi}


However, for input-output systems,
there is an additional wrinkle when defining the $\eps$-spectral value set abscissa and radius:
eigenvalues may not be of interest if they are not \emph{controllable} and \emph{observable}, 
concepts which we now define.

\begin{defi}
Let $\lambda$ be an eigenvalue of the matrix pencil $(A,E)$ from an input-output system.
Eigenvalue $\lambda$ is \emph{observable} if $Cx \ne 0$ holds for all of its right eigenvectors $x$,
i.e. $Ax = \lambda E x, x \ne 0$.
Eigenvalue $\lambda$ is \emph{controllable} if $B^*y \ne 0$ holds for all of its left eigenvectors $y$,
i.e. $y^*A = \lambda y^* E, y \ne 0$.
\end{defi}

In a sense, the presence of uncontrollable and/or unobservable eigenvalues can be considered an artifact 
of redundancy in a specific system design.
Any associated transfer function $G(\lambda)$ of \eqref{eq:lti_cont} with uncontrollable or unobservable eigenvalues
can be reduced to what is called a \emph{minimal realization} $\widehat G(\lambda)$, whose eigenvalues
are all controllable and observable; 
e.g. see \cite[Theorem 2-6.3]{Dai89}.
The $A$ and $E$ matrices of $\widehat G(\lambda)$ are of minimal possible dimension 
so that the reduced transfer function is unaltered and its input-output behavior remains identical to $G(\lambda)$.

In terms of spectral value sets, consider an eigenvalue $\lambda$ of $(A,E)$ 
with right and left eigenvectors $x$ and $y$.
If $\lambda$ is unobservable or uncontrollable, then $Cx = 0$ or $B^*y = 0$ respectively holds,
and thus for any perturbation matrix $\Delta \in \Cmn{m}{p}$,
either $M(\Delta) x = Ax$ or $y^*M(\Delta) = y^* A$ holds.
Furthermore, if $\lambda$ is a simple eigenvalue, 
then for sufficiently small $\eps > 0$, 
$\lambda$ must be an isolated point of $\SVSeps(A,B,C,D,E)$:
letting $\lambda(t)$ be some parameterization of $\lambda$ with $t \in \R$ and $\lambda(0) = \lambda$,
via standard perturbation theory for simple eigenvalues, 
it is easily seen that $\lambda^\prime(0) = 0$ holds.

Since the presence of uncontrollable/unobservable eigenvalues will only affect
the point in the complex plane used to initialize the algorithms presented here
(and such eigenvalues can be removed as a preprocessing step),
for the remainder of the paper
we simply assume whether or not they are included is determined by the user.


\subsection{Derivatives of the norm of the transfer function}
\label{sec:ntf_derivs}
As we will utilize first- and second-order information of 
$\|G(\lambda)\|_2$ in different ways,
we will need the following results.
For technical reasons, we will first need the following assumption.

\begin{assumption}
Let $\eps > 0$ with $\eps \|D\|_2 < 1$ and let $\lambda \in \SVSeps(A,B,C,D,E)$ with $\lambda \not\in \spec(A,E)$.
Then the largest singular value of $G(\lambda)$ is simple.
\end{assumption}

\begin{rema}
For almost all quintuplets $(A,B,C,D,E)$,  
the largest singular value of $G(\lambda)$ is indeed simple for all $\lambda \in \C\backslash\spec(A,E)$; 
e.g. see \cite[\S2]{BurLO03} for pseudospectra and \cite[Remark 2.20]{GugGO13} for general spectral value sets with $E=I$.  
Although counter examples can be constructed (see \cite[Remark 2.20]{GugGO13}),
with probability one such examples will not be encountered in practice and as such,
this technicality does not pose a problem for the algorithms we propose here.
\end{rema}

Let $\lambda(t) \in \C$ be parameterized with respect to $t \in \R$ and $Z(t) = \lambda(t)E - A$.
Then 
\beq
	\label{eq:GofLambda}
	(G \circ \lambda)(t) = C(\lambda(t)E - A)^{-1}B + D = C Z(t)^{-1} B + D.
\eeq
By standard (matrix) differentiation techniques, we have that:
\begin{subequations}
\label{eq:GofLambda_derivs}
\begin{align}
	\label{eq:GofLambda_deriv1}
	(G \circ \lambda)^\prime(t) {}& \coloneqq -\lambda^\prime(t) C Z(t)^{-1} E Z(t)^{-1} B \\
	\label{eq:GofLambda_deriv2}
	(G \circ \lambda)^{\prime\prime}(t) {}& \coloneqq
	2 \lambda^\prime(t)^2 C Z(t)^{-1} E Z(t)^{-1} E Z(t)^{-1} B - \lambda^{\prime\prime}(t) C Z(t)^{-1} E Z(t)^{-1} B.
\end{align}
\end{subequations}
Furthermore, let $s(t)$ be the largest singular value of \eqref{eq:GofLambda}, i.e. $\|G(\lambda(t))\|_2$, 
with associated left and right singular vectors $u(t)$ and $v(t)$.
Assuming that $s(t)$ is a simple singular value at say, $t=0$,
by standard perturbation theory, it follows that
\begin{subequations}
\label{eq:sigma_deriv1}
\begin{align}
	s^\prime(0) 
		&= \Real{ u(0)^* \left[ (G \circ \lambda)^\prime(0) \right] v(0) } \\
		&= - \Real{ u(0)^* \left[ \lambda^\prime(0) CZ(0)^{-1} E Z(0)^{-1} B \right] v(0) }.
\end{align}
\end{subequations}
To compute $s^{\prime\prime}(0)$, we need the following result for the second derivative of eigenvalues,
which can be found in various forms, e.g. \cite{Lan64,OveW95}.
\begin{theo}
\label{thm:eig2ndderiv}
For $t \in \R$, let $H(t)$ be a twice-differentiable $n \times n$ Hermitian matrix family with distinct eigenvalues at $t=0$ with $(\lambda_k,x_k)$ denoting the $k$th such eigenpair and where each eigenvector $x_k$ has unit norm
and the eigenvalues are ordered $\lambda_1 \geq \ldots \geq \lambda_n$.\footnote{
In \cite[Remark~4.2]{BenM18a}, we were overly cautious in assuming that all the singular values of $G(\lambda(t))$ at $t=0$ are simple; only simplicity of the largest singular value is needed.
}  
Then:
\[
	\lambda_1''(t) \bigg|_{t=0}= x_1^* H''(0) x_1 + 2 \sum_{k = 2}^{n} \frac{| x_1^* H'(0) x_k |^2}{\lambda_1 - \lambda_k}.
\]
\end{theo}
Since $s(t)$ is the largest singular value of $G(\lambda(t))$, it is also 
the largest eigenvalue of the matrix:
\beq
	\label{eq:eigderiv_mat}
	\eigderivmat(t) = 
	\begin{bmatrix} 0 & G(\lambda(t)) \\ 
				G(\lambda(t))^* & 0 
	\end{bmatrix},
\eeq
which has first and second derivatives
\beq
	\begingroup
	\setlength\arraycolsep{3pt}
	\label{eq:eigderiv12_mat}
	\eigderivmat^\prime(t) = 
	\begin{bmatrix} 0 & (G \circ \lambda)^\prime(t) \\ 
				(G \circ \lambda)^\prime(t)^* & 0
	\end{bmatrix}
	~~~\text{and}~~~
	\eigderivmat^{\prime\prime}(t) = 
	\begin{bmatrix} 0 & (G \circ \lambda)^{\prime\prime}(t) \\ 
				(G \circ \lambda)^{\prime\prime}(t)^* & 0
	\end{bmatrix},
	\endgroup
\eeq
and where the nonzero blocks are given by \eqref{eq:GofLambda_derivs}.
Thus, by constructing matrix \eqref{eq:eigderiv_mat} and its first and second derivatives given in \eqref{eq:eigderiv12_mat},
$s^{\prime\prime}(0)$ can be computed by a straightforward application of Theorem~\ref{thm:eig2ndderiv}.

Note that $s^\prime(0)$ and $s^{\prime\prime}(0)$ are relatively cheap to compute once
$s(0)$ has been.  An LU factorization needed to apply $Z(0)^{-1}$ for computing $s(0)$ 
can be saved and reused to cheaply compute the two matrices given in \eqref{eq:GofLambda_derivs}, 
noting that ignoring $\lambda^\prime(t)$, \eqref{eq:GofLambda_deriv1} appears in \eqref{eq:GofLambda_deriv2}.
Moreover, the eigenvectors $x_k$ of \eqref{eq:eigderiv_mat} can be  
obtained from the full SVD of $G(\lambda(0))$.
Let $\sigma_k$ be the $k$th singular value of $G(\lambda(0))$ with associated
right and left singular vectors $u_k$ and $v_k$, respectively.  
Then $\pm\sigma_k$ is an eigenvalue of \eqref{eq:eigderiv_mat} with  
eigenvector $\left[ \begin{smallmatrix} u_k \\ v_k \end{smallmatrix} \right]$ for $\sigma_k$ and 
eigenvector $\left[ \begin{smallmatrix} u_k \\ -v_k \end{smallmatrix}\right]$ for $-\sigma_k$.
The eigenvector for $\sigma_k = 0$ is either 
$\left[\begin{smallmatrix} u_k \\ \mathbf{0} \end{smallmatrix} \right]$
(when $p > m$) or
$\left[\begin{smallmatrix} \mathbf{0} \\ v_k \end{smallmatrix} \right]$
(when $p < m$),
where $\mathbf{0}$ denotes a column of $m$ or $p$ zeros, respectively.

When $B=C=I_n$ and $D=0$, 
the LU and backsolves can be completely avoided 
by instead equivalently computing the reciprocal of the smallest singular value of $\lambda E - A$.
Otherwise, if $G(\lambda)$ will be evaluated at more than just a handful of points, 
making LU factorizations of $\lambda E - A$ for each $\lambda \in \C$ 
can also be inefficient.  As shown in \cite{Lau81} for $E=I$,
one can first make an upper Hessenberg factorization of $A = PHP^*$, which is $\bigO(n^3)$ work but
only needs to be done once.  Then $G(\lambda)$ can be evaluated as
$\widetilde C (\lambda I - H)^{-1} \widetilde B  + D$,
where $\widetilde C = CP$ and $\widetilde B = P^*B$;
applying  $(\lambda I - H)^{-1}$ only requires $\bigO(n^2)$ work as
it remains Hessenberg for any $\lambda \in \C$.  
This Hessenberg technique also extends to when $E \neq I$ \cite{VanV85}.

\section{Directly extending the pseudospectral abscissa algorithm}
\label{sec:abs_old}
The criss-cross method of \cite{BurLO03} alternates between vertical and horizontal search phases, which we now describe and generalize to computing the spectral value set abscissa.

\subsection{Vertical search}
\label{sec:vertical_search}
The following fundamental theorem relates singular values of the transfer function,
evaluated at some point $\lambda \in \C$, 
to purely imaginary eigenvalues of an associated matrix pencil. 
A key tool for various stability measure algorithms,
 including the criss-cross method of \cite{BurLO03},
 this correspondence was first shown by \cite{Bye88} for $B=C=E=I$, $D=0$, and $x=0$ 
and has previously appeared in various less general specific extensions than what we present here.
We defer its proof, and that of the upcoming Theorem~\ref{thm:anyline_search}, to Appendix~\ref{apdx:proofs}.

\begin{theo}
\label{thm:eigsing_cont}
Let $x \in \R$, $y \in \R$, $\gamma > 0$ not a singular value of $D$, 
and $\lambda E - A$ be regular.  
Consider the matrix pencil $\MNpencont$, where
\beq
	\label{eq:MNpencil_cont}
	\begingroup
	\setlength\arraycolsep{2.5pt}
	\Mc \coloneqq 
	\begin{bmatrix} 	
	A - xE - BR^{-1}D^*C & -\gamma BR^{-1}B^* \\ 
	\gamma C^*S^{-1}C 	& -(A - xE - BR^{-1}D^*C)^* 
	\end{bmatrix},
	\ \ \ \
	\Nc \coloneqq  \begin{bmatrix} E & 0\\ 0 & E^*\end{bmatrix},
	\endgroup
\eeq
$R = D^*D - \gamma^2 I$, and $S = DD^* - \gamma^2 I$.
Then $\iy$ is an eigenvalue of $\MNpencont$ if and only if
$\gamma$ is a singular value of $G(\xiy)$ and $\xiy$ is not an eigenvalue of $(A,E)$.
\end{theo}

By setting $\gamma = \eps^{-1}$, Theorem~\ref{thm:eigsing_cont} immediately leads to the ability to compute all the boundary points, if any, of an 
$\eps$-spectral value set that lie on any desired vertical line specified by the value of $x$.
Given these boundary points, 
the subset of adjacent pairs on this vertical line which correspond to segments in the $\eps$-spectral value set can be determined in multiple ways.
While there are a few ways to do this, 
just evaluating the norm of the transfer function at their midpoints is a simple and robust choice.

\begin{rema}
\label{rem:pencil_nonsingular}
Note that the matrix pencil given by \eqref{eq:MNpencil_cont} cannot be singular.  
If it were, then $\gamma$ would be a singular value of $G(\xiy)$ for all $y \in \R$ 
and thus the entire vertical line specified by value $x$ would be a part of $\SVSeps(A,B,C,D,E)$.
Since $(A,E)$ is regular and $\eps$ is finite, this is not possible.
\end{rema}

\subsection{Horizontal search}
\label{sec:horizontal_search}
Given vertical line $x = \eta$,
let $\cs_k = [y_k,y_{k+1}]$ denote a cross section segment of $\SVSeps(A,B,C,D,E)$ on this line and
$\cs = \{\cs_1,\ldots,\cs_q\}$ denote the set of all such cross sections for $x=\eta$, 
with at least one having nonzero length.
Without loss of generality, assume that interval $\cs_k$ has nonzero length.
Since any point $\eta + \imagunit y$ with $y \in (y_k,y_{k+1})$ is strictly in the interior of $\SVSeps(A,B,C,D,E)$,
rightward progress within the spectral value set is indeed possible from vertical line $x=\eta$.

In \cite{BurLO03}, it was proposed to consider rightward progress from the midpoints 
of all the positive-length vertical cross sections $\cs_k \in \cs$, i.e., along horizontal lines given by $\psi_k = 0.5 ( y_k + y_{k+1})$.
The maximal rightward progress is then given by solving:
\beq
	\label{eq:horizontal_search}
	\max_{\cs_k \in \cs} \max \{\Re \lambda : \lambda \in \SVSeps(A,B,C,D,E) ~\text{and}~ \Im \lambda = \psi_k \}.
\eeq
To solve \eqref{eq:horizontal_search}, \cite{BurLO03} applied 
a ``rotated" version of Theorem~\ref{thm:eigsing_cont} to compute all boundary points 
along each horizontal line.  
We now present this result not only extended to spectral value sets
but also to any line in the complex plane, as this more general form will be used in \S\ref{sec:rad_old}
for computing the spectral value set radius.

\begin{defi}
Let $\theta \in [0,2\pi)$ denote the angle between the $x$-axis and some ray from the origin, 
with the positive $x$ and $y$ directions respectively given by $\theta=0$ and $\theta=\pi/2$.
Given $s\in\R$,
we define $L(\theta,s)$ as the parallel line to the left of the ray given by $\theta$,
separated by distance $s$, with left defined with respect to the direction $\theta$.
\end{defi}

\begin{theo}
\label{thm:anyline_search}
Given the line $L(\theta,s)$, 
let $\{\imagunit \omega_1,\ldots,\imagunit \omega_l\}$ be the set of purely imaginary eigenvalues of \eqref{eq:MNpencil_cont},
where $\gamma = \eps^{-1}$, $x=-s$, and matrices $A$ and $B$ have been respectively replaced by 
$e^{\imagunit \theta_\mathrm{r}} A$ and $e^{\imagunit \theta_\mathrm{r}}  B$, with $\theta_\mathrm{r} = \pi/2 - \theta$.
Then the points $\lambda_j = e^{-\imagunit \theta_\mathrm{r}}(-s + \imagunit \omega_j)$ define the cross sections
of $\SVSeps(A,B,C,D,E)$ along $L(\theta,s)$.
\end{theo}

By Theorem~\ref{thm:anyline_search},
the boundary points of  $\SVSeps(A,B,C,D,E)$ along the horizontal line $L(0,\psi_k)$
are given by $\omega_j + \imagunit \psi_k$,
where $\{\imagunit \omega_1,\ldots,\imagunit \omega_l\}$ are the imaginary eigenvalues (sorted in increasing  order) of 
the rotated version of \eqref{eq:MNpencil_cont} given by Theorem~\ref{thm:anyline_search}.
Thus, $\psi_k  + \imagunit \omega_l$ is the rightmost boundary point along line $L(0,\psi_k)$, with 
$\omega_l > \eta$ (assuming the corresponding cross section had positive length).
Applying this procedure to each of the horizontal lines given by the $\psi_k$ midpoints
yields the solution to \eqref{eq:horizontal_search}.

\subsection{The complete directly-extended abscissa algorithm}
\label{sec:abs_old_complete}
Computing the pseudospectral abscissa, 
as originally specified in \cite{BurLO03},
begins with a vertical search and then alternates between horizontal and vertical searches,
to respectively increase estimate $x = \eta$ (monotonically) and find the new vertical cross sections;
see Figure~\ref{fig:cc_abs} for a visualization of this process.
The procedure converges to a globally rightmost point $\lambda_\star \in \sigma_\eps(A)$, with 
$\eta_\star = \Re \lambda_\star = \aleps(A)$. 
A critical requirement for global convergence is that the initial vertical search must be 
done to the right of a globally rightmost eigenvalue of matrix $A$; 
it cannot be done exactly through an eigenvalue of $A$
as this would violate the conditions of Theorem~\ref{thm:eigsing_cont}
so in practice a small perturbation is used.
Under a regularity assumption, the authors showed that the criss-cross method has local
quadratic convergence \cite[\S5]{BurLO03}.
The algorithm requires $\bigO(n^3)$ work and $\bigO(n^2)$ memory per iteration, 
both with notably large constants since it must compute all the imaginary eigenvalues of $q+1$ different matrix pencils of size $2n \times 2n$ per iteration: one pencil for the vertical search and $q$ pencils for the corresponding $q$ cross-sections of positive length in the horizontal search phase.
Provided that a structure-preserving and backward-stable Hamiltonian eigenvalue solver is used
for both vertical and horizontal searches, the method has been shown to be backwards stable
\cite[\S2.1.2]{Men06}, which was done by combining an upper bound on the accuracy from horizontal searches with a lower bound on the accuracy from vertical searches.
Via Theorems~\ref{thm:eigsing_cont} and \ref{thm:anyline_search}, the extension to computing the spectral value set abscissa is clear, but, as noted in \cite[\S6]{BurLO03}, there
is one last subtlety that must be addressed to ensure a robust implementation in practice.

\begin{figure}
\center
\subfloat[Vertical and horizontal searches]{
\includegraphics[scale=.25,trim={0cm 0 0cm 0},clip]{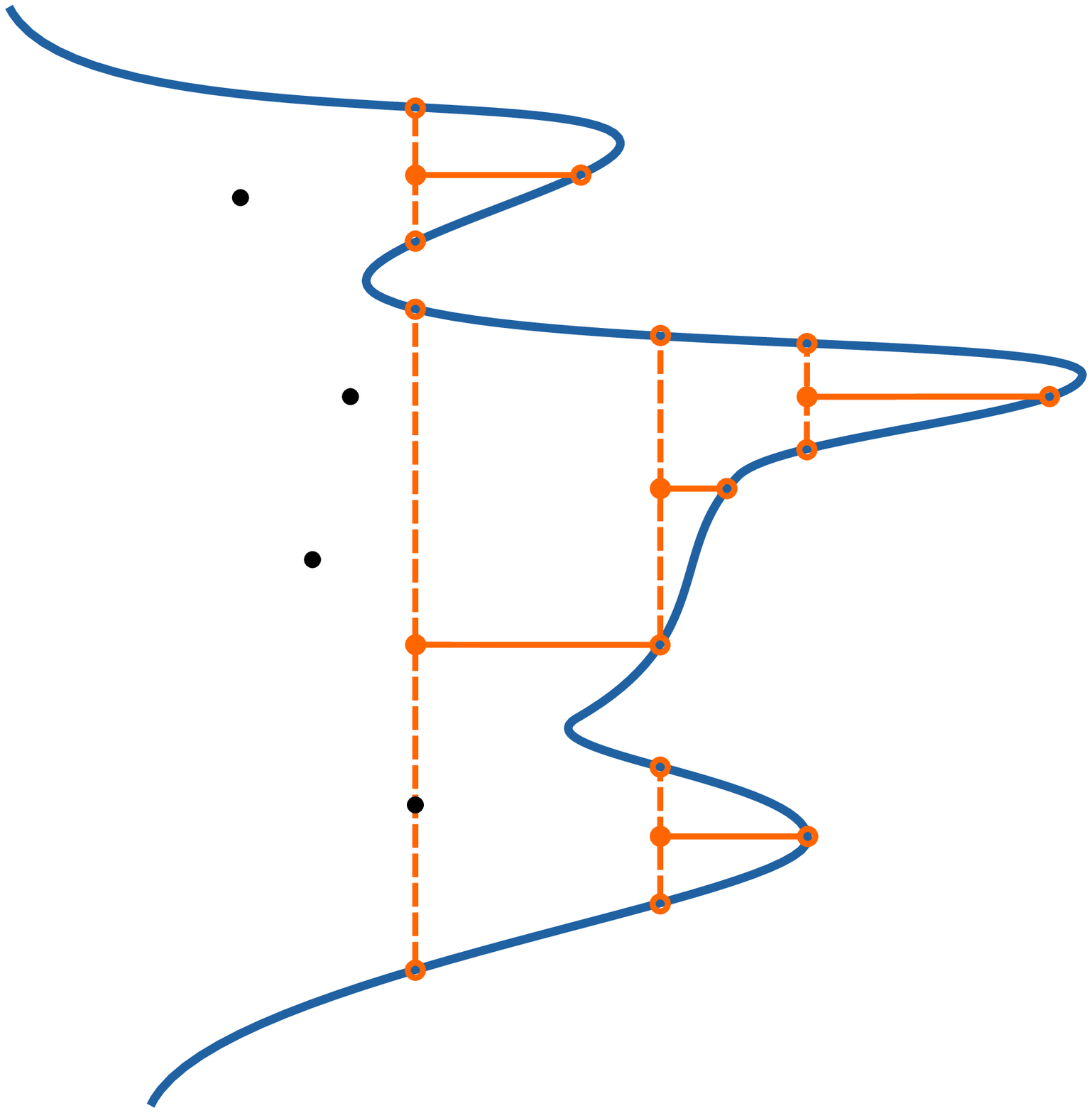} 
\label{fig:cc_abs}
}
\subfloat[Circular and radial searches]{
\includegraphics[scale=.25,trim={0cm 0 0cm 0},clip]{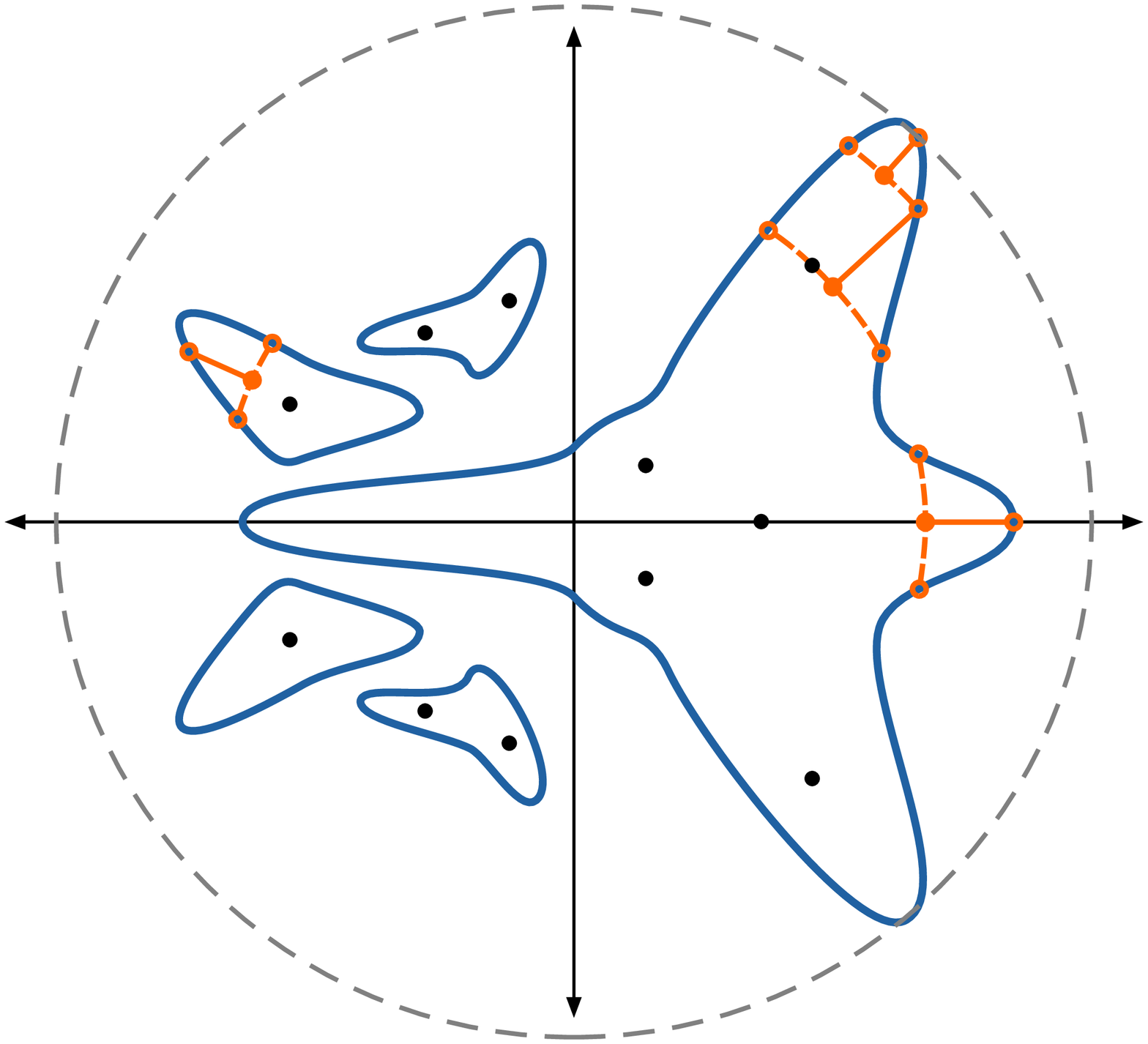}
\label{fig:cc_rad}
}
\caption{Illustrations of the iterations (shown in orange) of the directly-extended criss-criss methods.
The spectrum of $(A,E)$ and the spectral value set boundaries are respectively shown by black dots and blue contours.
For the abscissa case (left), the dashed vertical line segments with dots at their midpoints show the cross sections found by the vertical searches,
while the solid horizontal line segments depict the horizontal searches rightward.
For the radius case (right), the  dashed arcs, also with dots at their midpoints, show the results of the circular searches,
while the emanating rays depict the radial searches outward; 
the current best estimate of the spectral value set radius is shown by the grey dashed circle centered at the origin.
}
\label{fig:cc_diagrams}
\end{figure}

Suppose that a given vertical search passes through the interior of a spectral value set such 
that it intersects its boundary at three points and the resulting vertical cross-section intervals combined with the boundary to the right of them form an outline similar to the capital letter ``B".  
While the top and bottom boundary points on the vertical line will correspond to 
simple eigenvalues of \eqref{eq:MNpencil_cont}, the middle 
boundary point will correspond to a double eigenvalue.  
Due to rounding error, an eigensolver
may only return the upper and lower simple imaginary eigenvalues and fail to return
the double imaginary eigenvalue between them.
In this case, the algorithm would find a single vertical cross-section interval,
instead of the two actual adjacent intervals.  
If the missed double eigenvalue also happens to coincide with the midpoint 
of the larger computed interval, 
the subsequent (and only) horizontal search will not be able to make any rightward progress.
Hence, the algorithm will erroneously terminate there,
 failing to proceed to either of the  two locally rightmost points further to the right as it should.
 Per \cite[\S6]{BurLO03}, this pitfall can occur in practice, but remarkably, it 
 can be overcome via a simple fix: split any vertical cross-section 
into two whenever the previous best horizontal search
(i.e. the one that maximizes \eqref{eq:horizontal_search}) 
passes through it and is considered too close to its midpoint.
The previous horizontal search determines the split point and only
one interval may be split per vertical search, thus increasing the number of horizontal
searches incurred per iteration at most by one.

\begin{rema}
On \cite[p. 373]{BurLO03}, a second test is described to 
avoid splitting intervals too frequently (which would increase cost), based
on skipping the above safeguard whenever the intervals are deemed sufficiently small.
Oddly, this second test is not implemented in the authors' MATLAB routines,
although their comments in the code clearly refer to it as well.
Robustly implementing such a ``bypass" test also seems difficult:
given a problem which requires the safeguard to prevent stagnation,
the problematic cross-section interval can always be shrunk to any desired length by simply rescaling the entire problem, thus ensuring stagnation for this rescaled problem.
\end{rema}

\section{The improved abscissa algorithm}
\label{sec:abs_new}
While computing all the imaginary eigenvalues of  $\MNpencont$, the pencil given by the matrices in \eqref{eq:MNpencil_cont}, can be quite expensive,
merely computing the norm of the transfer function can be done much faster.
As we reported in \cite[Table 2]{BenM18a},
this performance gap ranged from up to 2.47 to 119 times faster for various random dense examples.
Thus, there is a potential to increase efficiency by working with 
the transfer function as opposed to Hamiltonian eigenvalue problems wherever possible.
This is particularly so for input-output systems, where $m,p \ll n$ is typical,
since computing the eigenvalues of $\MNpencont$ is relatively unaffected by dimensions $m$ and $p$
while the norm of the transfer function (a $p \times m$ matrix) becomes significantly cheaper to compute for small
$m,p$.  
Since obtaining all vertical cross sections on every iteration is necessary to ensure global convergence,
which as far as is known, must be done via computing all the eigenvalues of $\MNpencont$,
we instead focus on avoiding the difficult and expensive eigenvalue problems in the horizontal search phases.
As we will see in \S\ref{sec:numerical}, this approach can be several times faster than the directly-extended method
and also reduces numerical inaccuracies due to rounding errors in the eigensolves (see Figure~\ref{fig:eig_errors}).

\subsection{Horizontal searches via root finding}
Instead of using Theorem~\ref{thm:anyline_search} for horizontal searches,
as a first step in exploiting the above cost disparity,
we propose to find boundary points further to the right via root finding using the norm of the transfer function.
With $y$ specifying a fixed vertical position, 
we parameterize the largest singular value of the transfer function by the horizontal position $x$:
\begin{equation}
	\label{eq:ntfy_of_x}
	\ntfy(x) \coloneqq \| G(\lambda_y(x)) \|_2 = \| CZ_y(x)^{-1}B + D \|_2,
\end{equation}
where $\lambda_y(x) \coloneqq  x + \imagunit y$ and $Z_y(x) \coloneqq  \lambda_y(x) E - A$.
Then, with $\cs_k$ defining a cross section of nonzero length along vertical line $x=\eta$,
with midpoint $\eta + \imagunit \psi_k$,
the globally rightmost point of $\SVSeps(A,B,C,D,E)$ along line $L(0,\psi_k)$
is given by the rightmost root $x_\star$ of 
\beq
	\label{eq:horizontal_root}
	\ntfy(x) - \eps^{-1} = 0,
\eeq
with $x_\star = \omega_l > \eta$.
Since \eqref{eq:horizontal_root} may have more than one root to the right of line $x=\eta$,
rooting finding will not always guarantee that $x_\star$ is found.
However, for $\eta$ sufficiently close to the value of the $\eps$-spectral value set abscissa, 
$x_\star$ will be the only remaining root of \eqref{eq:horizontal_root} to the right
and so this will not be a problem as the algorithm converges.
Furthermore, if we maintain updating lower and upper bounds $x_\mathrm{lb} \ge \eta$ and $x_\mathrm{ub}$ 
such that $\eqref{eq:horizontal_root}$ is always positive at $x_\mathrm{lb}$ and 
always negative at $x_\mathrm{ub}$,
at least always finding a root $\tilde x_\star > \eta$ of \eqref{eq:horizontal_root} in bracket 
$(x_\mathrm{lb},x_\mathrm{ub})$
such that $\tilde x_\star + \imagunit \psi_y$ is also a locally rightmost point of 
$\SVSeps(A,B,C,D,E) \cap L(0,\psi_k)$ is guaranteed.
This bracketing scheme precludes the obviously suboptimal possibility of converging to 
a root $\tilde x_\star > \eta$ such that $\tilde x_\star + \imagunit \psi_y$ is a locally \emph{leftmost} point of 
$\SVSeps(A,B,C,D,E) \cap L(0,\psi_k)$.

While the current horizontal position $x=\eta$ always provides an initial lower bound,
an upper bound must be found iteratively but this is always doable; 
since $\lim_{x \to \infty} \ntfy(x) = \|D\|_2$ and $\eps \|D\|_2 < 1$,
\eqref{eq:horizontal_root} converges to some negative value as $x \to \infty$.
Furthermore, by exploiting first and possibly second derivatives of singular values, 
a hybrid Newton- or Halley-based root-finding method enforcing
our above bracketing convergence criteria could be employed; 
near roots of \eqref{eq:horizontal_root}, quadratic or cubic convergence would be expected.
Since the first and second derivatives of \eqref{eq:horizontal_root} are also relatively cheap to obtain 
compared to \eqref{eq:horizontal_root} itself (as discussed in \S\ref{sec:ntf_derivs}), 
it stands to reason that computing a root $\tilde x_\star > \eta$ may be
significantly faster than computing $x_\star$ via Theorem~\ref{thm:anyline_search},
at least for all but the smallest of problems.  

The first and second derivatives of \eqref{eq:ntfy_of_x} are as follows.
Suppose that $\ntfy(\hat x)$ is a simple singular value with associated left and right singular vectors $\hat u$ and $\hat v$.
As $\lambda_y^\prime(x) = 1$ and $\lambda_y^{\prime\prime}(x) = 0$, by \eqref{eq:GofLambda_derivs},
it follows that
\begin{subequations}
\label{eq:tfx_derivs}
\begin{align}
	\label{eq:tfx_deriv1}
	(G \circ \lambda_y)^\prime (x) &= - CZ_y(x)^{-1} E Z_y(x)^{-1} B \\
	\label{eq:tfx_deriv2}
	(G \circ \lambda_y)^{\prime\prime}(x) &= 2 C Z_y(x)^{-1} E Z_y(x)^{-1} E Z_y(x)^{-1} B.
\end{align}
\end{subequations}
Again by \eqref{eq:sigma_deriv1}, the first derivative of \eqref{eq:ntfy_of_x} at $\hat x$ is
\beq
	\label{eq:ntfx_deriv1}
	\ntfy^\prime(\hat x) = -\Real{ \hat{u}^* C Z_y(\hat x)^{-1} E Z_y(\hat x)^{-1} B \hat v },
\eeq
while its second derivative at $\hat x$ can be computed via applying Theorem~\ref{thm:eig2ndderiv} 
to matrix \eqref{eq:eigderiv_mat} with first and second derivatives \eqref{eq:eigderiv12_mat},
respectively defined by $G(\lambda_y(\hat x))$ and the matrix derivatives given in \eqref{eq:tfx_derivs} evaluated at $\hat x$.

\begin{rema}
While \cite[Theorem 4.1]{BurLO03} also considered the pseudospectral analogues of the 
first derivative given in \eqref{eq:ntfx_deriv1}, 
interestingly it was only used
analytically and not computationally to improve efficiency and accuracy, as we do here.
\end{rema}

We forgo further root-finding details and instead specify the following 
abstract function that we will utilize as a subroutine in our improved method.

\begin{defi}
\label{def:findrootroutine}
Let \emph{\texttt{[$r$,$\delta$] = findARootToTheRight($f(\cdot)$,$x_0$)}} be some
routine that implements the bracketing scheme described above, which
given a function $f(\cdot)$ and an initial guess $x_0$ with $f(x_0) > 0$, 
returns a root $r$ of $f(\cdot)$ such that $r > x_0$ and
$f(r + \mu) < 0$ for all $\mu \in (0,\beta)$ for some fixed value $\beta > 0$.
In inexact arithmetic, only $f(r) \approx 0$ is guaranteed and $r + \delta$ would have been the next Newton/Halley iterate.
\end{defi}

\subsection{Intelligently ordering the horizontal searches}
The second way we exploit the increased efficiencies of working with the transfer function 
is by intelligently ordering the $q$ horizontal searches on each iteration,
so that we solve the most promising ones first, 
i.e., the ones likely to provide the most rightward progress in the spectral value set.
Then, provided the predicted ordering is sufficiently accurate, already computed solutions to the more promising root problems can be leveraged to cheaply determine when solving the subsequent root problems is either not necessary or to at least warm start the computations.
Our exact procedure works as follows.

Observe that the left side of \eqref{eq:horizontal_root} provides a measure of distance between a
point $x + \imagunit y$ and the boundary of $\SVSeps(A,B,C,D,E)$.  
It thus stands to reason that a global optimizer of \eqref{eq:horizontal_search} 
might most likely lie on the particular horizontal line $\imagunit \psi_k$ that 
maximizes $\ntfx[\eta](y)$ for $y \in \{\psi_1,\ldots,\psi_q\}$.
However, we have found that using either
\beq
	\label{eq:priority}
	N_k = - \frac{\ntfy[\psi_k](\eta) - \eps}{\ntfy[\psi_k]^\prime(\eta)} 
	\quad \text{or} \quad 
	H_k = - \frac{
		2 \cdot \ntfy[\psi_k](\eta) \cdot \ntfy[\psi_k]^\prime(\eta) 
		}{
		2 \cdot [\ntfy[\psi_k]^\prime(\eta) ]^2 - \ntfy[\psi_k](\eta) \cdot \ntfy[\psi_k]^{\prime\prime}(\eta)
		},
\eeq
the initial Newton/Halley steps for each of the horizontal root-finding subproblems, often produces an even better ordering.  Before solving any of the root problems for the $q$ horizontal searches, we compute,
say $H_k$, at each of the $\eta + \imagunit \psi_k$ midpoints,
and then reorder the searches so that they are in descending order with respect to their $H_k$ values.
For convenience, assume that the $\psi_k$ midpoints are already in this order.
Let $\tilde x_\star$ be a computed root of \eqref{eq:horizontal_root} for $y = \psi_1$, 
which had nothing but $\eta$ to use as a starting point.
We then warm start the second solve, \eqref{eq:horizontal_root} for $y = \psi_2$, 
using $\tilde x_\star$ as a starting point.
If the left side of \eqref{eq:horizontal_root} is negative at $\tilde x_\star$,
we immediately have an upper bound on a root for $\psi_2$ 
that is worse (to the left) than root $\tilde x_\star$ for $\psi_1$ and 
we have no evidence that there are any other roots to the right; 
hence, we can completely skip solving \eqref{eq:horizontal_root} for $\psi_2$ 
and instead proceed to $\psi_3$.
Similarly, if the left side of \eqref{eq:horizontal_root} is exactly zero at $\tilde x_\star$, 
then \eqref{eq:horizontal_root} is already solved but does not yield a better root so we again proceed to $\psi_3$ without further computation.
However, if the left side of \eqref{eq:horizontal_root} for $\psi_2$ is positive, 
then solving \eqref{eq:horizontal_root} for $\psi_2$ initialized at $\tilde x_\star$ must yield a root $\hat x_\star > \tilde x_\star$,
and so the solve should proceed.  
We continue to warm start the subsequent subproblems with the current best root, 
a clearly better strategy than solving them all initialized at $x = \eta$.

Finally, if the rightmost computed approximate root of \eqref{eq:horizontal_root} 
ends up corresponding to a point just inside the spectral value set, 
we slightly perturb it so that it is just outside, by adding a multiple of the final Newton/Halley 
step.  
This slight modification prevents the algorithm from incurring two vertical searches just before termination
when only one is necessary numerically.
The full pseudocode is given in Subroutine~\ref{alg:fastsearch}.
While we have not used parallelism in our description,
it could potentially further improve running times.

\begin{algfloat}[t]
\begin{subroutine}[H]
\caption{\texttt{[$x$,$\psi$] = fastSearch($\eta_0$,$\{\psi_1,\ldots,\psi_q\}$)}}
\label{alg:fastsearch}
\begin{algorithmic}[1]
	\REQUIRE{ \quad \\
		$\eta_0 \in \R$, an initial lower bound \\ 
		$\{\psi_1,\ldots,\psi_q\}$ with each $\psi_k \in \R$ \\ 
	}
	\ENSURE{ \quad \\
		$x > \eta$, the rightmost root \emph{encountered} of 
			\eqref{eq:horizontal_root} over $y \in \{\psi_1,\ldots,\psi_q\}$ or $\eta = \eta_0$
		\\ \quad
	}
	
	\STATE compute the initial Newton (or Halley) steps $\{N_1,\ldots,N_q\}$ using \eqref{eq:priority}
		\label{algline:initial_steps}
	\STATE let $\{\psi_1,\ldots,\psi_q\}$ be ordered such that $N_k$ is decreasing with respect to all $k$
		\label{algline:midpoint_sort_fn}
	\STATE $x \coloneqq \eta_0$; $\psi \coloneqq []$; $\delta \coloneqq []$
	\FOR { $k = 1,\ldots,q$ }
		\STATE set function handle $f(\cdot) \coloneqq \ntfy[\psi_k](\cdot) - \eps^{-1}$ \COMMENT{\eqref{eq:horizontal_root} defined for $y=\psi_k$}
			\label{algline:root_fn}
		\IF { $f(x) > 0$ }
			\STATE \texttt{[$x$,$\delta$] $\coloneqq$ findARootToTheRight($f(\cdot)$,$x$)}
			\STATE $\psi \coloneqq \psi_k$
		\ENDIF
	\ENDFOR 
	\STATE \COMMENT{Ensure computed (approximate) rightmost root is not just inside the interior  } 
	\IF { $\delta$ is \NOT [] \AND $f(x) > 0$ }
		\STATE $\delta \coloneqq \min \{| \delta | , | x | \cdot \emach \}$ 
			\COMMENT{Make sure $x + \delta > x$ holds to machine precision}
		\STATE $x \coloneqq \min \{ x + k\delta : f(x + k\delta) \le 0 \text{ for } k \in \{1,2,\ldots\} \}$
	\ENDIF
\end{algorithmic}
\end{subroutine}
\end{algfloat}

\subsection{The complete improved abscissa algorithm}
\label{sec:svs_abs_alg}
Naturally, we advocate using \texttt{fastSearch} (Subroutine~\ref{alg:fastsearch}) in lieu of the earlier 
expensive eigenvalue-based horizontal searches.
However, we also propose one last but simple modification: to start with 
a horizontal search rather than a vertical one.  This has two benefits.
First, it often reduces the total number of horizontal searches incurred.
The number of vertical cross sections generally decreases as $\eta \to \eta_\star$ and 
since there can be up to $n$ cross sections, which is more likely when $\eta \ll \eta_\star$, 
the reduction in horizontal searches can be dramatic (though our \texttt{fastSearch} subroutine
generally only resolves at most a handful of roots when given many searches).
Moreover, by having a better (larger) initial estimate of $\eta_\star$, the number of vertical searches may also be reduced.
While most of our efficiency gains will be from \texttt{fastSearch},
this additional change also has non-negligible effect
(see Table~\ref{table:svs_abs} in \S\ref{sec:numerical}). 
The second benefit is that it eliminates the need for an ad hoc perturbation 
to do the first vertical search just to the right of an eigenvalue $(A,E)$,
since doing it through an eigenvalue would violate the assumptions 
of Theorem~\ref{thm:eigsing_cont}.
Initialization is typically done from a controllable and observable eigenvalue of $(A,E)$; in practice, it is best to provide a minimal realization as input.
Pseudocode for the complete improved method is given in Algorithm~\ref{alg:crisscross}.

\begin{algfloat}[t]
\begin{algorithm}[H]
\caption{\texttt{[$\eta$] = svsAbscissa($\eps$,$A$,$B$,$C$,$D$,$E$)}}
\label{alg:crisscross}
\begin{algorithmic}[1]
	\REQUIRE{ \quad \\
		$\eps > 0$ with $\eps \|D\|_2 < 1$ and matrices $A$, $B$, $C$, $D$, $E$ defining $\SVSeps(A,B,C,D,E)$
	}
	\ENSURE{ \quad \\
		$\eta$, the computed value of $\aleps(A,B,C,D,E)$ 
		\\ \quad
	}

	\STATE $\Lambda \coloneqq$ \texttt{eig($A$,$E$)}
	\STATE $\Lambda \coloneqq \{ \lambda \in \Lambda : \lambda ~\text{meets user's inclusion criteria: controllable/observable}\}$
	\IF {$ \infty \in \Lambda$ } 
		\RETURN $\eta = \infty$ 
	\ENDIF \label{alg:cc_pre2}
	\STATE $ \lambda_0 \coloneqq \argmax \{ \Re \lambda : \lambda \in \Lambda\}$
	\STATE \COMMENT{More efficient to start with a horizontal search instead of a vertical one:}
	\STATE\texttt{[$\eta$,$y$]} $\coloneqq$ \texttt{fastSearch($\Re \lambda_0$,$\{  \Im \lambda_0 \}$)}
	\WHILE { $\eta < \aleps(A,B,C,D,E)$ }
		\STATE compute imaginary eigenvalues $\{\imagunit y_1,\ldots,\imagunit y_l\}$ 
				of \eqref{eq:MNpencil_cont} for $x=\eta$ and $\gamma = \eps^{-1}$
		\STATE form all intervals $\cs_k = [y_k,y_{k+1}]$ s.t. $\eta + \imagunit y \in \SVSeps(A,B,C,D,E) ~\forall y \in \cs_k$
		\STATE $\Psi \coloneqq \{\psi_1,\ldots,\psi_q\}$ such that $\psi_k$ is a midpoint of interval $\cs_k$ 
		\STATE \texttt{[$\eta$,$y$]} $\coloneqq$ \texttt{fastSearch($\eta$,$\Psi$)}
	\ENDWHILE
\end{algorithmic}
\end{algorithm}
\end{algfloat}

\section{Directly extending the pseudospectral radius algorithm}
\label{sec:rad_old} 
We now describe and directly extend the pseudospectral radius method of \cite{MenO05} to the spectral value set radius, by generalizing its alternating circular and radial search phases.  Although respectively 
analogous to the vertical and horizontal searches described in \S\ref{sec:abs_old}, 
a key difference here is that circular searches may sometimes fail, 
either because the corresponding pencils are singular or the searches do not intersect with the spectral value set  boundary. 
For now, we assume neither of these happen.

\subsection{Circular search}
\label{sec:circular_search}
We now give an analogue of Theorem~\ref{thm:eigsing_cont}
that relates singular values of the norm transfer function evaluated 
at points on a chosen circle of radius $r > 0$ centered at the origin
with unimodular eigenvalues of an associated sympletic pencil.
Less general versions of this result go back as far as \cite[\S3]{HinS91}, 
for the special case of fixed radius $r=1$, $D=0$, and $E=I$.
Similarly to Theorem~\ref{thm:eigsing_cont}, we defer the proof to Appendix~\ref{apdx:proofs}.

\begin{theo}
\label{thm:eigsing_disc}
Let $r > 0$ be the radius of a circle centered at the origin, angle $\theta \in [0,2\pi)$,
$\gamma > 0$ not a singular value of $D$, and $\lambda E - A$ be regular.
Consider the matrix pencil $\MNpendisc$, where
\beq
	\label{eq:MNpencil_disc}
	\begin{aligned}
	\Md \coloneqq {}& \begin{bmatrix} 	A - BR^{-1}D^*C 	& -\gamma BR^{-1}B^* \\ 
										0 				& rE^* \end{bmatrix}, \\
	\Nd \coloneqq {}& \begin{bmatrix} 	rE & 0\\ 
									-\gamma C^*S^{-1}C 	& A^* - C^*DR^{-1}B^* \end{bmatrix},
	\end{aligned}
\eeq
$R = D^*D - \gamma^2 I$ and $S = DD^* - \gamma^2 I$.
Then $\eitheta$ is an eigenvalue of $\MNpendisc$ if and only if
$\gamma$ is a singular value of $G(r\eitheta)$ and $r\eitheta$ is not an eigenvalue of $(A,E)$.
\end{theo}

Setting $\gamma = \eps^{-1}$, Theorem~\ref{thm:eigsing_disc} provides a means to compute all the boundary points, if any, 
of an $\eps$-spectral value set that lie on any desired circle of radius $r$ centered at the origin.
More specifically, let $\{\theta_1,\ldots,\theta_l\}$ be the set of angles, all in $[0,2\pi)$ and sorted in increasing order, 
given by the (we assume nonempty) set of unit-modulus eigenvalues of \eqref{eq:MNpencil_disc}.
Thus, each point $r e^{\imagunit \theta_j}$ is a boundary point of $\SVSeps(A,B,C,D,E)$. 
As in \S\ref{sec:vertical_search}, determining the subset of arcs on the circle of radius $r$ that pass through the spectral value set 
can be reliably done by just evaluating the norm of the transfer function at the midpoints of all the candidate arc segments
given by $\cs_k = [\theta_k,\theta_{k+1}]$ for $k = 1,\ldots,l-1$,
but now
the additional ``wrap-around" interval $[\theta_l,\theta_1 + 2\pi]$ must be also be considered.

\subsection{Radial search}
\label{sec:radial_search}
Given a circle of radius $r = \eta$,
let $\cs_k = [\theta_k,\theta_{k+1}]$ denote a non-zero length arc of this circle which also lies in $\SVSeps(A,B,C,D,E)$ and
\mbox{$\cs = \{\cs_1,\ldots,\cs_q\}$} denote the set of all such arcs.
Similar to the abscissa case, \cite{MenO05} proposed to make outward progress by 
taking the midpoints of these arc segments, i.e. $\psi_k = 0.5 ( \theta_k + \theta_{k+1})$, 
as directions of rays from the origin on which to find more distant boundary points.
The maximal outward progress is then:
\beq
	\label{eq:radial_search}
	\max_{\cs_k \in \cs} \max \{| \lambda | : \lambda \in \SVSeps(A,B,C,D,E) ~\text{and}~ \Arg \lambda = \psi_k \},
\eeq
which can be solved by applying Theorem~\ref{thm:anyline_search} to each of the lines $L(\psi_k,0)$
and taking the outermost of all the computed boundary points.
Of course, we will instead adapt our new faster \texttt{fastSearch} subroutine; see \S\ref{sec:rad_new}.

\subsection{The complete directly-extended radius algorithm}
The method of \cite{MenO05} alternates between 
radial and circular searches to respectively increase estimate $r = \eta$ (monotonically)
and find new arc-shaped cross sections of the pseudospectrum.
A robust implementation also requires the splitting safeguard 
described at the end of \S\ref{sec:abs_old_complete}. 
It  converges to a globally outermost point $\lambda_\star$ of 
$\sigma_\eps(A)$ with \mbox{$\eta_\star = | \lambda_\star | = \rhoeps(A)$},
with a local quadratic convergence rate \cite[\S2.4]{MenO05};
a sample of the iterations is depicted visually in Figure~\ref{fig:cc_rad}.
However, global convergence is not just predicated upon initializing the algorithm 
with an initial radius $r \ge \rho(A)$;
the method must also handle the aforementioned possibility of circular searches failing.
This problem was dealt with in \cite[\S2.5]{MenO05} in the following manner. 
We first present respective generalizations of \cite[Theorem~2.11 and Corollary~2.12]{MenO05};
the proofs extend directly via simple substitutions.

\begin{theo}
\label{thm:singular_pencils}
Given some $r>0$, if the matrix pencil defined by \eqref{eq:MNpencil_disc} is singular and the 
largest singular value of $G(r\eitheta)$ is simple for all $\theta \in [0,2\pi)$, then either:
\begin{enumerate}
\item the boundary of $\SVSeps(A,B,C,D,E)$ contains the circle of radius $r$ or
\item the circle of radius $r$ is strictly inside $\SVSeps(A,B,C,D,E)$.
\end{enumerate}
\end{theo}

\begin{coro}
\label{cor:avoid_singular}
Suppose that for some fixed $r > 0$, $\|G(r\eitheta)\|_2 - \eps^{-1} < 0$ holds for at least one angle $\theta \in [0,2\pi)$.  
Then the matrix pencil defined by \eqref{eq:MNpencil_disc} is regular.
\end{coro}

First, \cite{MenO05} proposed starting the algorithm with a single radial search along the ray from the origin through
a globally rightmost eigenvalue $\lambda_0$ of $A$. 
By applying Theorem~\ref{thm:anyline_search} to find $\lambda_\mathrm{bd}$, a globally outermost point  of $\SVSeps(A) \cap L(\Arg \lambda_0,0)$ or the solution of \eqref{eq:radial_search} in the direction of $\lambda_0$,
Corollary~\ref{cor:avoid_singular} asserts that the matrix pencil given by \eqref{eq:MNpencil_disc} is regular for all $r > | \lambda_\mathrm{bd}|$.
Moreover, since the corresponding circular searches must then always have portions outside of $\SVSeps(A)$, 
they are also guaranteed not to be problematic interior searches.
However, in exact arithmetic, the possibility of an initial circular search with radius $|\lambda_\mathrm{bd}|$ 
corresponding to a singular pencil cannot be ruled out.
Furthermore, the computed version of $\lambda_\mathrm{bd}$, which we denote $\tilde \lambda_\mathrm{bd}$,
may be strictly inside $\SVSeps(A)$ and so the possibility that a circular search of radius $| \tilde \lambda_\mathrm{bd}|$ does not intersect with
the pseudospectral boundary cannot be ruled out either.
Thus, \cite{MenO05} also proposed potentially increasing the radius of the very first circular search 
from $|\tilde \lambda_\mathrm{bd}|$ to $|\tilde \lambda_\mathrm{bd}| + k \delta_0$,
where $\delta_0$ is the initial Newton step to change the magnitude of $\tilde \lambda_\mathrm{bd}$ in order to move it to the pseudospectral boundary and $k$ is the smallest nonnegative integer such that adding $k \delta_0$ to its magnitude indeed puts the resulting point outside out of $\SVSeps(A)$.
When $\tilde \lambda_\mathrm{bd}$ is strictly inside the pseudospectrum, 
$\delta_0 > 0$ holds and the authors noted that small $k$ (e.g. 1 or 2) typically sufficed
to move $\tilde \lambda_\mathrm{bd}$ outside; otherwise $k=0$ is taken.\footnote{In fact, this is essentially the same perturbation procedure we have employed 
at the end of \texttt{fastSearch} 
but motivated by very different reasons.}
By Corollary~\ref{cor:avoid_singular}, it is not necessary to perturb any subsequent 
circular searches but there is a caveat.  If the perturbation is too 
small the resulting pencil may be \emph{nearly} singular and thus still problematic to solve (which we have observed in practice), or alternatively, the perturbation is large but then accuracy may be sacrificed.
While this procedure extends to the directly-extended spectral value set radius algorithm, 
it does not for our improved radius algorithm.

\begin{rema}
In \cite{MenO05}, starting with a radial search 
only seems to be for avoiding singular pencils;
no mention is made that it can also have efficiency benefits.
\end{rema}

\section{The improved radius algorithm}
\label{sec:rad_new}
Before describing \texttt{fastSearch} for the radial phases, 
to efficiently find locally-optimal solutions of \eqref{eq:radial_search},
note that the loss of global optimality in these searches violates
the necessary assumptions to use the existing technique of \cite{MenO05} for handling
singular pencils and/or interior circular searches.
We now adapt \texttt{fastSearch} and then
propose a new compatible technique to overcome such difficult pencils/searches.

\subsection{Adapting \texttt{fastSearch} for the radial phase}
Parameterizing the largest singular value of the transfer function 
in polar coordinates, with varying radius $r$ for a \emph{fixed angle} $\theta$, yields
\begin{equation}
	\label{eq:ntft_of_r}
	\ntft(r) \coloneqq \| G(\lambda_\theta(r)) \|_2 = \| CZ_\theta(r)^{-1}B + D \|_2,
\end{equation}
where $\lambda_\theta(r) \coloneqq r\eitheta$ and $Z_\theta(r) \coloneqq \lambda_\theta(r) E - A$.
Hence each outward search along a ray in direction $\theta$ is 
done by finding a root of
\beq
	\label{eq:radial_root}
	\ntft(r) - \eps^{-1} = 0.
\eeq
The first and second derivatives of \eqref{eq:ntft_of_r} are as follows.
Assume that $\ntft(\hat r)$ is a simple singular value, with left and right singular vectors $\hat u$ and $\hat v$.
As $\lambda_\theta^\prime(r) = \eitheta$ and $\lambda_\theta^{\prime\prime}(r) = 0$, 
by \eqref{eq:GofLambda_derivs}, it follows that 
\begin{subequations}
\label{eq:tfr_derivs}
\begin{align}
	\label{eq:tfr_deriv1}
	(G \circ \lambda_\theta)^\prime(r) &= - \eitheta CZ_\theta(r)^{-1} E Z_\theta(r)^{-1} B \\
	\label{eq:tfr_deriv2}
	(G \circ \lambda_\theta)^{\prime\prime}(r) &= 2 e^{2\imagunit \theta} C Z_\theta(r)^{-1} E Z_\theta(r)^{-1} E Z_\theta(r)^{-1} B,
\end{align}
\end{subequations}
and so by \eqref{eq:sigma_deriv1}, the first derivative of \eqref{eq:ntft_of_r} at $\hat r$ is
\beq
	\label{eq:ntfr_deriv1}
	\ntft^\prime(\hat r)  = -\Real{ \eitheta \hat{u}^* C Z_\theta(\hat r)^{-1} E Z_\theta(\hat r)^{-1} B \hat v }.
\eeq
Using \eqref{eq:tfr_derivs}, the second derivative of \eqref{eq:ntft_of_r} at $\hat r$ can be computed via Theorem~\ref{thm:eig2ndderiv}.
The subproblems given by \eqref{eq:radial_root} are prioritized in descending order with respect to
their initial Newton/Halley steps, i.e. \eqref{eq:priority} with $\ntfy[\psi_k](\eta)$ replaced by $\ntft[\psi_k](\eta)$.
Thus remaining modification to \texttt{fastSearch} replaces \eqref{eq:horizontal_root} with \eqref{eq:radial_root} 
in Subroutine~\ref{alg:fastsearch}.

\begin{rema}
Similar to \cite{BurLO03}, 
the pseudospectral analogue of the 
first derivative given in \eqref{eq:ntfr_deriv1}
was considered in \cite[Theorem 2.3]{MenO05}, but 
it too was not used computationally to improve efficiency, as we do here.
\end{rema}

\subsection{A new method for handling singular pencils and interior searches}
\label{sec:singular_pencils}
Although \texttt{fastSearch} is guaranteed to converge 
to a global solution of \eqref{eq:radial_search}
for $\eta$ sufficiently close to $\eta_\star$, 
it may only return locally-optimal solutions 
for smaller values of $\eta$.
Consequently, and in contrast to the directly-extended algorithm,
we cannot rule out the possibility of encountering a (nearly) singular pencil 
or problematic interior search on \emph{any} iteration.
It might seem tempting to just apply the perturbation technique of \cite{MenO05}
to every iteration, but this comes with the accuracy-versus-reliability tradeoff
mentioned above.  However, since \texttt{fastSearch} finds boundary points 
to high accuracy, $\delta_0$ will generally be tiny, 
meaning that using the earlier singular pencil procedure of \cite{MenO05} would almost always result 
in pencils that are still nearly singular.
The technique of \cite{MenO05} is reasonable for the directly-extended algorithm 
because a) its $\delta_0$ value is generally much larger due to the relatively higher inaccuracy 
of obtaining the solution to \eqref{eq:radial_search} via computing eigenvalues 
(Figure~\ref{fig:eig_errors} demonstrates such errors) and 
b) it is only needed once rather than multiple times (where the chance of encountering a single failure increases significantly).
Faced with such difficulties, we consider an entirely new approach, using the following new result.

\begin{theo}
\label{thm:avoid_singular}
Given $\eps > 0$ with $\eps \|D\|_2 < 1$, set $\gamma = \eps^{-1}$ and 
let $\eta$ be such that the circle of radius $\eta$ centered at the origin
both encircles all the eigenvalues of $(A,E)$ and is strictly in the interior of $\SVSeps(A,B,C,D,E)$.
Let $\delta > 0$ be the largest value such that, for all $t\in [0,1]$, circles of radius $\eta + t\delta$
are still subsets of $\SVSeps(A,B,C,D,E)$.
Finally, let $R = \{r_1,\ldots,r_l\}$ denote the subset of positive radii corresponding to the
boundary points of $\SVSeps(A,B,C,D,E)$ that lie on $L(\theta,0)$ but are outside the circle of radius $\eta$, 
where $\theta \in [0,2\pi)$ has been chosen randomly.
Then for $\hat r = \min \{r_1,\ldots,r_l\} > \eta$, either of the two following scenarios may hold:
\begin{enumerate}
\item $\MNpendisc$ is singular for $r=\eta+\delta$ 
	but \mbox{$\hat r = \eta + \delta = \rhoeps(A,B,C,D,E)$} or
\item $\MNpendisc$ is regular for $r = \eta + \delta$ and, with probability one, $\hat r  > \eta + \delta$.
\end{enumerate}
\end{theo}
\begin{proof}
We first consider the case where $\MNpendisc$ is singular at $r = \eta + \delta$.
Since $\eta + \delta \in R$, it must be that $r\eitheta$ is a boundary point of $\SVSeps(A,B,C,D,E)$,
and by Theorem~\ref{thm:singular_pencils},
the circle of radius $r$ centered at the origin must be a subset of the boundary of $\SVSeps(A,B,C,D,E)$.
Furthermore, $\rhoeps(A,B,C,D,E) \ge r$.  
If strict inequality holds, then there must exist some boundary point $\hat\lambda \in \SVSeps(A,B,C,D,E)$ with $|\hat\lambda| > r$.
But this contradicts the conclusion of Lemma~\ref{lem:path},
that there exists a path taking some controllable and observable eigenvalue of $(A,E)$
to $\hat\lambda$ such that only $\lambda(1) = \hat \lambda$ is a boundary point,
since any such $\lambda(t)$ must also cross the circle of radius $r$ at some $t < 1$.
Hence, $r=\hat r$ as $R$ only contains a single unique value, namely $\eta+\delta$.

Now suppose $\MNpendisc$ is regular at $r = \eta + \delta$.
By assumption, 
the circle of radius $\eta + \delta$ only touches the spectral value boundary but does not cross it.
Furthermore, since the pencil is regular, by Theorem~\ref{thm:eigsing_disc},
there can only be a finite number (at most $n$) of contact points between this circle and the spectral value set boundary.
Suppose that $\hat r = \eta + \delta$, noting that by assumption, $\hat r$ cannot be any smaller.   
Then, for boundary point $\hat r\eitheta$, 
its angle $\theta$ must be equal to one of the angles corresponding to the finite set of contact points.
As $\theta \in [0,2\pi)$ was chosen randomly, the probability of this event occurring is zero.
Therefore, with probability one, $\theta$ will not correspond to any of 
the contact points on the circle of radius $\eta + \delta$ and thus, $\hat r > \eta + \delta$.
\end{proof}

Theorem~\ref{thm:avoid_singular} clarifies what to do when \texttt{fastSearch} returns a point 
$\tilde \lambda_\mathrm{bd}$ with $| \tilde \lambda_\mathrm{bd} | = \eta$ such that, 
due to rounding error, $\tilde \lambda_\mathrm{bd}$ is distance $\delta$ inside $\SVSeps(A,B,C,D,E)$ and
the subsequent circular search for $r = \eta$ returns no arc cross sections.
Either the algorithm has actually converged to $\rhoeps(A,B,C,D,E)$
within the numerical limits of the root-finding method itself 
or, by reapplying \texttt{fastSearch} in a random direction given by $\theta \in [0,2\pi)$ through interior point $\eta \eitheta$,
the algorithm can, with probability one, obtain a new more distant point beyond the 
current problematic local area involving singular pencils and/or interior searches.
Put more simply, problematic circular searches encountered on any iteration can be overcome by applying \texttt{fastSearch}
in one or more random directions and if the algorithm still converges to a singular pencil, then with probability one
it has also converged to $\rhoeps(A,B,C,D,E)$. 

\subsection{The complete improved radius algorithm}
Like our improved abscissa method, the improved radius algorithm
uses \texttt{fastSearch} but now adapted for the radial searches.
It starts with a single initial radial search outward, from an outermost eigenvalue of $(A,E)$ (typically controllable and observable),
and then alternates between circular and radial searches.
However, whenever a circular search does not return any arc cross sections of $\SVSeps(A,B,C,D,E)$,
which generally would be a sign of convergence, 
the new algorithm must also consider the possibility that the search simply failed.
Thus, whenever no arc cross sections are obtained, 
the improved algorithm simply applies \texttt{fastSearch} in one or more randomly chosen directions in $[0,2\pi)$
to distinguish between convergence and encountering interior searches or (nearly) singular pencils.
If the algorithm has indeed converged, calling \texttt{fastSearch} has no effect, 
except for the relatively small additional cost
to evaluate the norm of the transfer function at a handful of random points.
If the algorithm has not converged, then by Theorem~\ref{thm:avoid_singular},
with probability one the method is guaranteed to increase its current estimate of $\rhoeps(A,B,C,D,E)$ 
beyond the problematic region.
Furthermore, as along as \emph{any} outward progress is being made,
\texttt{fastSearch} will continued to be called with new random directions every iteration
until either a subsequent circular search returns
one or more arc cross sections or \texttt{fastSearch} can no longer increase the radius estimate at all.
This allows the algorithm to robustly push past problematic regions where successive circular searches may fail to return any arcs.  
However, we have observed that typically only a single attempt is necessary in practice.  
Pseudocode for the complete improved radius method is given in Algorithm~\ref{alg:crisscrossradius}.

\begin{algfloat}[t]
\begin{algorithm}[H]
\caption{\texttt{[$\eta$] = svsRadius($\eps$,$A$,$B$,$C$,$D$,$E$)}}
\label{alg:crisscrossradius}
\begin{algorithmic}[1]
	\REQUIRE{ \quad \\
		$\eps > 0$ with $\eps \|D\|_2 < 1$ and matrices $A$, $B$, $C$, $D$, $E$ defining $\SVSeps(A,B,C,D,E)$ \\
		$r$ a positive integer, specifying how many random angles to try 
	}
	\ENSURE{ \quad \\
		$\eta$, the computed value of $\rhoeps(A,B,C,D,E)$, with probability one
		\\ \quad
	}
	
	\STATE $\Lambda \coloneqq$ \texttt{eig($A$,$E$)}
	\STATE $\Lambda \coloneqq \{ \lambda \in \Lambda : \lambda ~\text{meets user's inclusion criteria: controllable/observable}\}$
	\IF {$ \infty \in \Lambda$ } 
		\RETURN $\eta = \infty$ 
	\ENDIF 
	\STATE $ \lambda_0 \coloneqq \argmax \{ |\lambda| : \lambda \in \Lambda\}$
	\STATE $\Psi \coloneqq \{\Arg \lambda_0, \psi_1,\ldots,\psi_r\}$ such that $\psi_k$ is chosen randomly from $[0,2\pi)$
	\STATE \texttt{[$\eta$,$\theta$]} $\coloneqq$ \texttt{fastSearch($|\lambda_0|$,$\Psi$)}
	\WHILE { $\eta < \aleps(A,B,C,D,E)$ }
		\STATE compute unimodular eigenvalues $\{e^{\imagunit \theta_1},\ldots,e^{\imagunit \theta_l}\}$ 
				of \eqref{eq:MNpencil_disc} for $r=\eta$ and $\gamma = \eps^{-1}$
		\STATE form all intervals $\cs_k = [\theta_k,\theta_{k+1}]$ s.t. $\eta \eitheta \in \SVSeps(A,B,C,D,E) ~\forall \theta \in \cs_k$
		\IF {no such intervals}
			\STATE $\Psi \coloneqq \{\psi_1,\ldots,\psi_r\}$ such that $\psi_k$ is chosen randomly from $[0,2\pi)$ 
		\ELSE
			\STATE $\Psi \coloneqq \{\psi_1,\ldots,\psi_q\}$ such that $\psi_k$ is a midpoint of interval $\cs_k$ 
		\ENDIF	
		\STATE \texttt{[$\eta$,$\theta$]} $\coloneqq$ \texttt{fastSearch($\eta$,$\Psi$)}
	\ENDWHILE
\end{algorithmic}
\end{algorithm}
\end{algfloat}

\section{Global Convergence}
\label{sec:convergence}
We give the following proof of convergence, which is simpler and less technical than those given in 
\cite{BurLO03} and \cite{MenO05}.

\begin{theo}
Algorithms~\ref{alg:crisscross} and \ref{alg:crisscrossradius} converge to the 
$\eps$-spectral value set abscissa and radius, respectively.
\end{theo}
\begin{proof}
Let $\eta_\star$ be the value of the $\eps$-spectral value set abscissa/radius,
attained at some non-isolated globally rightmost (outermost) point $\lambda_\star$,
and $\{\eta_k\}$ be the sequence our methods generate,
which by construction must be monotonically increasing
and $\eta_k \le \eta_\star$ must hold.
Let $\lambda(t)$ be one of the continuous paths, specified by Lemma~\ref{lem:path},
taking an eigenvalue of $(A,E)$ to $\lambda_\star$ with $\mathcal{N}(t) \subset \SVSeps(A,B,C,D,E)$ a neighborhood of $\lambda(t)$ of radius $\delta(t) > 0$ for all $t \in [0,1)$.
Setting $f(t) = \Re \lambda(t)$ ($f(t) = | \lambda(t) |$),
there exists $t_0 \in [0,1)$ such that $f(t_0) = \eta_0$.
So suppose that
$\eta_k \to \hat\eta < \eta_\star$.
By Theorem~\ref{thm:avoid_singular}, encountering singular pencils can be ruled out since this only occurs
in the radius case when $\hat \eta = \eta_\star$.
Let $\cs(\eta)$ denote the set of intervals corresponding to vertical (circular) cross sections varying by $x =\eta$ ($r=\eta$) and consider:
\[
	l(\eta) \coloneqq \max_{\cs_k \in \cs(\eta) } \{ | \omega_{k+1} - \omega_{k} | : \cs_k = [\omega_k, \omega_{k+1}] \}.
\]
As singular pencils are excluded, by continuity of eigenvalues, $l(\eta)$ must be continuous on $[\eta_0,\hat \eta]$.
Given the (possibly disjoint) subset $D \subset [t_0,1)$ where $f(t)$ is strictly increasing,
there exists $\hat t \in D$ such that $f(\hat t) = \hat \eta$.
Thus $l(f(t)) \ge \delta(t) > 0$ holds for all $t \in D$ and $l(\eta) \not\to 0$ as $\eta \to \hat \eta$.
This implies continuous convergence to a cross section of positive length at $\hat \eta$,
whose midpoint 
must of course be strictly in the interior of the spectral value set.
Hence, the methods cannot stagnate at $\hat \eta$.
\end{proof}

\section{Implementation}
\label{sec:code}
We implemented Algorithms~\ref{alg:crisscross} and \ref{alg:crisscrossradius} in a single new \matlab\ routine called \texttt{specValSet}, which is publicly available as part of the open-source library ROSTAPACK: RObust STAbility PACKage, starting with the v2.0 release.\footnote{
\url{http://timmitchell.com/software/ROSTAPACK}
}
For the radius case, 
whenever no intervals are obtained, \texttt{specValSet}
generates three random angles for \texttt{fastSearch} for invoking Theorem~\ref{thm:avoid_singular}.
By default, all evaluations of the norm of the transfer function are done using the Hessenberg factorization techniques
of \cite{Lau81,VanV85} mentioned at the end of \S\ref{sec:ntf_derivs}, though \texttt{specValSet} also supports using LU factorizations.

An evaluation of which bracketing and root-finding method would be most efficient to use 
for implementing the prerequisite subroutine \texttt{findARootToTheRight} (specified in Definition~\ref{def:findrootroutine})
is beyond the scope of this article.
We implemented a second-order version of \texttt{findARootToTheRight}.
It first brackets a root by iteratively increasing the current
guess by adding the larger of either two times the absolute value of the Halley step 
or the distance from the current guess and the initial guess $x_0$, until an upper bound has been found 
(while increasing the lower bound along the way).
Then it computes a root using a hybrid Halley's method with our bracketing.
We found that this was generally more efficient than using first-order schemes.
The very first step of the upper bound search increases the initial guess by at least $\max \{10^{-6},0.01|x_0| \}$.
If the function given to \texttt{findARootToTheRight} fails to return a finite value,
our code simply updates the lower bound and increases the current guess.

As a practical optimization, for when all the matrices are real valued but $\lambda_0$ is not, 
\texttt{specValSet} always attempts to first find a root along the $x$-axis, 
either to the right of $\lambda_0$ (or outward in either direction for the $\eps$-spectral value set radius)
before computing a solution to the root problem for $\lambda_0$.
Assuming such a root exists along the $x$-axis, 
the initial $\eps$-spectral value set abscissa (or radius) estimate $\eta$ will be increased, 
from $\eta = \alpha(A,E)$ (or $\eta = \rho(A,E)$)
to some larger value corresponding to a boundary point $\SVSeps\abcde$ on the $x$-axis.
Even though this strategy potentially introduces an additional horizontal search (or two radial searches), 
it often substantially reduces the overall number of \emph{complex-valued} SVDs incurred,
replacing them with \emph{much cheaper real-valued ones}.  
This optimization can have a significant net benefit in terms of running time because it can sometimes require many iterations 
to find an upper bound for the root-finding problem for $\lambda_0$, 
which without this optimization, would be initialized at $\lambda_0$, a pole of the transfer function.

The \texttt{specValSet} routine has the following similarities to the \texttt{pspa} and \texttt{pspr} routines of \cite{robstabcontmengi},
the respective implementations of the original criss-cross type methods for 
computing the pseudospectral abscissa \cite{BurLO03} and 
the pseudospectral radius \cite{MenO05}.
First, if the problem is real valued, the spectral value sets are symmetric with respect to the $x$-axis;
in this case, any interval $\Omega_k \in \Omega$ that corresponds to a section in the open lower half-plane
is discarded (since it is ``duplicated" by its positive conjugate).
Second, as \texttt{pspr} does not use a structure-preserving eigensolver, 
we used \texttt{eig} from \matlab\ for all codes in the benchmarks done here;
note that any robust implementation should use structure-preserving eigensolvers,
such as those available in SLICOT \cite{BenMSetal99}. 
Third, \texttt{specValSet} simply terminates when the $\eps$-spectral value set abscissa/radius estimate $\eta$ 
can no longer be increased, by any amount; no tolerance is needed.

\section{Numerical experiments}
\label{sec:numerical}
All experiments were done in \matlab\ R2017b on a laptop 
with an Intel i7-6567U dual-core CPU, 16GB of RAM, and macOS v10.14.
Running times were measured using \texttt{tic} and \texttt{toc}; to account for variability,
we report the average time of five trials for each method-problem pair.
For \texttt{specValSet}, we used ROSTAPACK v2.2 and set \texttt{rng(100)} before each trial (as it uses random numbers).

\subsection{Spectral value set evaluation}
\label{sec:num_svs}
We used 15 publicly-available spectral value set test examples of varying dimensions: 
four problems (CBM, CM3, CM4, CSE2) from \cite{GugGO13} and 
another 11 from the SLICOT benchmark examples.\footnote{
Available at \url{http://slicot.org/20-site/126-benchmark-examples-for-model-reduction}
}
Since some of the examples have nonzero $D$ matrices, and $\eps \|D\|_2 < 1$ must hold,
we instead calculated specific per-problem values of $\eps$ as follows.
We computed the continuous- and discrete-time $\Linf$ norms for each example, 
via \texttt{getPeakGain} with a tolerance of $10^{-14}$,
to be respectively used for the $\eps$-spectral value set abscissa and radius evaluations.
Let $\gamma_\star$ denote the corresponding computed $\Linf$-norm value, for
either the abscissa or radius case.
We then set $\eps \coloneqq 2\gamma_\star$, 
provided that $\gamma_\star$ was a finite positive value and $\eps \|D\|_2 < 0.5$ held.
Otherwise, for problems with nonzero $D$ matrices, we used $\eps \coloneqq 0.5 \|D\|_2^{-1}$ 
and $\eps \coloneqq 0.01$ for the rest.
Each problem was initialized at a rightmost/outermost controllable and observable eigenvalue of $(A,E)$.

\begin{table}[!t]
\small
\setlength{\tabcolsep}{3.5pt}
\robustify\bfseries
\center
\begin{tabular}{ l | ccc |cc | cc | cc | cl | @{}S@{}S | c } 
\toprule
\multicolumn{15}{c}{Spectral Value Set Abscissa: directly extended versus new method} \\
\midrule
\multicolumn{1}{c}{} & \multicolumn{3}{c}{Dimensions} &
	\multicolumn{4}{c}{\# solves} & \multicolumn{4}{c}{\# searches} & 
	\multicolumn{3}{c}{} \\
\cmidrule(lr){2-4}
\cmidrule(lr){5-8}
\cmidrule(lr){9-12}
\multicolumn{1}{l}{Problem} & $n$ & $m$ & \multicolumn{1}{c}{$p$} &
	\multicolumn{2}{c}{Eig} & \multicolumn{2}{c}{SVD} & 
	\multicolumn{2}{c}{vert.} & \multicolumn{2}{c}{horz.} & 	
	\multicolumn{2}{c}{time (sec.)} & \multicolumn{1}{c}{\% faster}\\
\midrule
\texttt{build}         &    48 &   1 &   1 &  19 &    4 &   25 &   57 &   6 &   4 &  13 &     4(4) &            0.113 &  \bfseries 0.049 &   132 \\
CSE2                   &    63 &   1 &  32 &   5 &    1 &    2 &   10 &   3 &   1 &   2 &     1(2) &            0.047 &  \bfseries 0.024 &    99 \\
\texttt{pde}           &    84 &   1 &   1 &   5 &    1 &    2 &    8 &   3 &   1 &   2 &     1(2) &            0.105 &  \bfseries 0.034 &   210 \\
\texttt{CDplayer}      &   120 &   2 &   2 &  10 &    3 &   13 &   36 &   4 &   3 &   6 &     3(4) &            0.252 &  \bfseries 0.086 &   192 \\
CM3                    &   123 &   1 &   3 &   6 &    2 &    8 &   64 &   3 &   2 &   3 &     2(2) &            0.253 &  \bfseries 0.106 &   139 \\
\texttt{heat-cont}     &   200 &   1 &   1 &   5 &    1 &    2 &   33 &   3 &   1 &   2 &     1(1) &            0.552 &  \bfseries 0.091 &   504 \\
\texttt{heat-disc}     &   200 &   1 &   1 &   5 &    1 &    2 &    8 &   3 &   1 &   2 &     1(1) &            0.972 &  \bfseries 0.190 &   410 \\
\texttt{random}$^*$    &   200 &   1 &   1 &   6 &    2 &    8 &   87 &   3 &   2 &   3 &     2(2) &            0.664 &  \bfseries 0.279 &   138 \\
CM4                    &   243 &   1 &   3 &  10 &    2 &   17 &   48 &   3 &   2 &   7 &     2(2) &            1.785 &  \bfseries 0.304 &   487 \\
\texttt{tline}$^*$     &   256 &   2 &   2 &   9 &    2 &   11 &   33 &   3 &   2 &   6 &     2(2) &            6.377 &  \bfseries 1.722 &   270 \\
\texttt{iss}           &   270 &   3 &   3 &   6 &    4 &    7 &   47 &   2 &   4 &   4 &     4(5) &            1.565 &  \bfseries 0.675 &   132 \\
\texttt{beam}$^*$      &   348 &   1 &   1 &   9 &    2 &   12 &   25 &   4 &   2 &   5 &     2(2) &            3.157 &  \bfseries 0.481 &   556 \\
CBM                    &   351 &   1 &   2 &   8 &    3 &   12 &   63 &   3 &   3 &   5 &     4(5) &            2.978 &  \bfseries 0.913 &   226 \\
\texttt{eady}          &   598 &   1 &   1 &   7 &    1 &    5 &    7 &   3 &   1 &   4 &     1(2) &            8.221 &  \bfseries 0.987 &   733 \\
\texttt{fom}           &  1006 &   1 &   1 &  12 &    3 &   15 &   31 &   4 &   3 &   8 &     4(6) &           61.139 &  \bfseries 8.852 &   591 \\
\midrule
\multicolumn{4}{r|}{Totals:} & 122 & 32 & 141 & 557 & \multicolumn{6}{r|}{Average \% faster:} & 321 \\ 
\midrule
\multicolumn{14}{r|}{(Directly extended with horz. search first) Average \% faster:} & 254 \\ 
\bottomrule
\end{tabular}
\caption{
For each pair of columns, performance data is given
for the directly-extended (DE) approach (left) and our improved approach (right)
for computing the spectral value set abscissa.
Problems marked with asterisks denote where the DE variant had
relative errors greater than $10^{-10}$.
The ``Eig" column gives the total number of $2n \times 2n$ eigensolves computed, while the ``SVD" 
column gives the total number of evaluations of the norm of the transfer function.
The number of vertical and horizontal searches are given under the ``vert." and ``horz." headers, respectively;
for our new method, 
if the total number of ``horz." searches was greater than the number that \emph{actually needed to be solved},
the latter is given first, with the former given in parenthesis. 
The time for the faster of the two methods is in bold.
Positive percentages in the ``\% faster" column indicate the amount faster
our new method was compared to the DE variant 
while negative ones indicate the amount faster 
the DE variant was.
The last row gives the average of the \% faster values for a second verson of DE
that starts with a horizontal search instead of a vertical one.
}
\label{table:svs_abs}
\end{table}

As our improved methods are the first to be able to compute the $\eps$-spectral value set 
abscissa and radius, there are no other available codes for comparison.  
Instead, we compared against our own implementations of the directly-extended (DE) variants 
described in \S\ref{sec:abs_old} and \S\ref{sec:rad_old} in order to show the benefits of our modifications.
While the values computed by both variants generally agreed with each other,
there were three examples, all abscissa problems, where the DE methods incurred relative errors
greater than $10^{-10}$ in magnitude; 
 these are marked with asterisks in Table~\ref{table:svs_abs}.
 Before analyzing these errors, we first present the performance results.

For the $\eps$-spectral value set abscissa tests, shown in Table~\ref{table:svs_abs},
we compared against two versions of the DE approach: 
one using a vertical search first and an alternative using an initial horizontal search,
though we only provide detailed per-problem performance statistics for the former. 
Overall, our method was much faster than the DE variant using a vertical search first:  
on average, our method was 321\% faster and up to 733\% faster (on \texttt{eady}).
In fact, our new approach was fastest on all 15 problems, all by significant margins;
even on CSE2, where performance difference was smallest, the DE variant
required about twice as much time.
Compared to the DE variant using an initial horizontal search, 
our new approach was still 254\% faster on average, 
underscoring that the majority of acceleration achieved is due to our new root-finding-based method and 
not just the simple (though beneficial) idea of starting with a horizontal search.

\begin{table}[!t]
\small
\setlength{\tabcolsep}{3.5pt}
\robustify\bfseries
\center
\begin{tabular}{ l | ccc |cc | cc | cc | cl | @{}S@{}S | c } 
\toprule
\multicolumn{15}{c}{Spectral Value Set Radius: directly extended versus new method} \\
\midrule
\multicolumn{1}{c}{} & \multicolumn{3}{c}{Dimensions} &
	\multicolumn{4}{c}{\# solves} & \multicolumn{4}{c}{\# searches} & 
	\multicolumn{3}{c}{} \\
\cmidrule(lr){2-4}
\cmidrule(lr){5-8}
\cmidrule(lr){9-12}
\multicolumn{1}{l}{Problem} & $n$ & $m$ & \multicolumn{1}{c}{$p$} &
	\multicolumn{2}{c}{Eig} & \multicolumn{2}{c}{SVD} & 
	\multicolumn{2}{c}{circ.} & \multicolumn{2}{c}{rad.} & 
	\multicolumn{2}{c}{time (sec.)} & \multicolumn{1}{c}{\% faster}\\
\midrule
\texttt{build}         &    48 &   1 &   1 &   6 &    3 &   13 &   34 &   3 &   3 &   3 &     4(5) &            0.042 &  \bfseries 0.031 &    36 \\
CSE2                   &    63 &   1 &  32 &   4 &    2 &    8 &   22 &   2 &   2 &   2 &     2(2) &            0.051 &  \bfseries 0.036 &    41 \\
\texttt{pde}           &    84 &   1 &   1 &   6 &    1 &    7 &   13 &   2 &   1 &   4 &     1(2) &            0.158 &  \bfseries 0.034 &   366 \\
\texttt{CDplayer}      &   120 &   2 &   2 &   2 &    1 &    5 &   18 &   1 &   1 &   1 &     1(1) &            0.090 &  \bfseries 0.051 &    76 \\
CM3                    &   123 &   1 &   3 &   4 &    2 &    9 &   54 &   2 &   2 &   2 &     2(2) &            0.208 &  \bfseries 0.133 &    56 \\
\texttt{heat-cont}     &   200 &   1 &   1 &   2 &    1 &    2 &   20 &   1 &   1 &   1 &     1(1) &            0.402 &  \bfseries 0.197 &   104 \\
\texttt{heat-disc}     &   200 &   1 &   1 &   3 &    1 &    3 &   12 &   1 &   1 &   2 &     1(2) &            0.803 &  \bfseries 0.193 &   317 \\
\texttt{random}        &   200 &   1 &   1 &   2 &    1 &    2 &   11 &   1 &   1 &   1 &     1(1) &            0.421 &  \bfseries 0.253 &    66 \\
CM4                    &   243 &   1 &   3 &   4 &    1 &    8 &   34 &   2 &   1 &   2 &     1(1) &            1.192 &  \bfseries 0.430 &   178 \\
\texttt{tline}         &   256 &   2 &   2 &  31 &    4 &   60 &  189 &   2 &   4 &  29 &    4(35) &           22.486 &  \bfseries 3.148 &   614 \\
\texttt{iss}           &   270 &   3 &   3 &   6 &    3 &   13 &   37 &   3 &   3 &   3 &     3(3) &            2.350 &  \bfseries 1.225 &    92 \\
\texttt{beam}          &   348 &   1 &   1 &   2 &    1 &    3 &   10 &   1 &   1 &   1 &     1(1) &            1.284 &  \bfseries 0.872 &    47 \\
CBM                    &   351 &   1 &   2 &   2 &    1 &    2 &   11 &   1 &   1 &   1 &     1(1) &            1.399 &  \bfseries 0.889 &    57 \\
\texttt{eady}          &   598 &   1 &   1 &   2 &    1 &    5 &   10 &   1 &   1 &   1 &     0(0) &            9.435 &  \bfseries 7.617 &    24 \\
\texttt{fom}           &  1006 &   1 &   1 &   2 &    1 &    2 &   20 &   1 &   1 &   1 &     1(1) &           63.508 & \bfseries 55.808 &    14 \\
\midrule
\multicolumn{4}{r|}{Totals:} & 78 & 24 & 142 & 495 & \multicolumn{4}{c}{} & \multicolumn{2}{r|}{Average \% faster:} & 139 \\ 
\bottomrule
\end{tabular}
\caption{
The headers remain mostly as described in Table~\ref{table:svs_abs},
except instead, in the $\eps$-spectral value set radius,
the number of circular and radial searches are given under the ``circ." and ``rad." headers, respectively.
}
\label{table:svs_rad}
\end{table}

In Table~\ref{table:svs_rad}, the corresponding experiments are shown 
for the $\eps$-spectral value set radius tests. 
Again our method was fastest on all 15 test problems; 
on average it was 139\% faster than the DE variant and 
up to 614\% faster (on \texttt{tline}).
However, on \texttt{eady}, and \texttt{fom}, the performance gains were rather small (14\% and 24\% faster, respectively).
The less pronounced performance gains against the DE variant on the radius problems seem 
to be due to the fact that, on average, 
the DE variant converged in fewer iterations for the radius case than it did for the abscissa case.

Returning to the three abscissa problems where the DE variants had the highest errors, 
all were caused by rounding errors when computing the imaginary eigenvalues in the \texttt{eig}-based
horizontal and/or vertical searches.  
On both \texttt{beam} and \texttt{random} ($1.19 \times 10^{-9}$ and $1.60 \times 10^{-9}$ relative errors, respectively), rounding errors in computing 
imaginary eigenvalues for the final horizontal search
caused the computed abscissa values to be slightly too large.
In contrast, the relative error of  $-2.38 \times 10^{-8}$ on \texttt{tline} was due to
rounding errors in both the horizontal and vertical searches.
On the last horizontal search, rounding errors in the computed imaginary eigenvalues caused the computed boundary point to be slightly inside the spectral value set.  Another vertical search was then attempted but failed to return any boundary points, 
again due to rounding errors in computing the imaginary eigenvalues.
Hence, the DE variant stopped a bit short of the spectral value set abscissa.
This particular failure underscores the importance of using structure-preserving eigensolvers 
for the vertical (and circular searches), even with our more numerically reliable root-finding-based approach.
However, as we will see in \S\ref{sec:num_ps}, even structure-preserving eigensolvers are not a panacea 
for the numerical issues that can arise in eigenvalue-based searches.


\subsection{Pseudospectral evaluation}
\label{sec:num_ps}
We also evaluated our new methods against
the \texttt{pspa} and \texttt{pspr} codes, using matrices of order 200 from EigTool~\cite{EigTool},  with \mbox{$\eps=0.01$}.
For brevity, we defer the full performance tables and detailed discussion to Appendix~\ref{apdx:pseudo}
but we note that our new method was on average 190\% and 84\% faster for the pseudospectral abscissa and radius cases, respectively.  
Only on one example, \texttt{orrsommerfeld\_demo} for the abscissa case, did \texttt{pspa}
incur a significant relative error ($1.75 \times 10^{-9}$).
Rounding errors in the eigenvalue value computations caused the 
horizontal searches to repeatedly overshoot the true abscissa value;
\texttt{pspa} not only incurred more iterations than necessary, it did so while also making its accuracy even worse. 
In Figure~\ref{fig:eig_errors}, we show an example of this phenomenon 
when computing the pseudospectral abscissa of \texttt{orrsommerfeld\_demo(201)}
for $\eps=10^{-4}$, where the relative error was even more pronounced: $7.75 \times 10^{-7}$.
When replacing \texttt{eig} by a structure-preserving eigensolver from SLICOT,
the relative error from \texttt{pspa} was $-4.68\times10^{-9}$.

\begin{figure}[!t]
\center
\subfloat[\texttt{pspa} with \texttt{eig}]{
\includegraphics[scale=.2395,trim={0.8cm 1cm 1.86cm 1cm},clip]{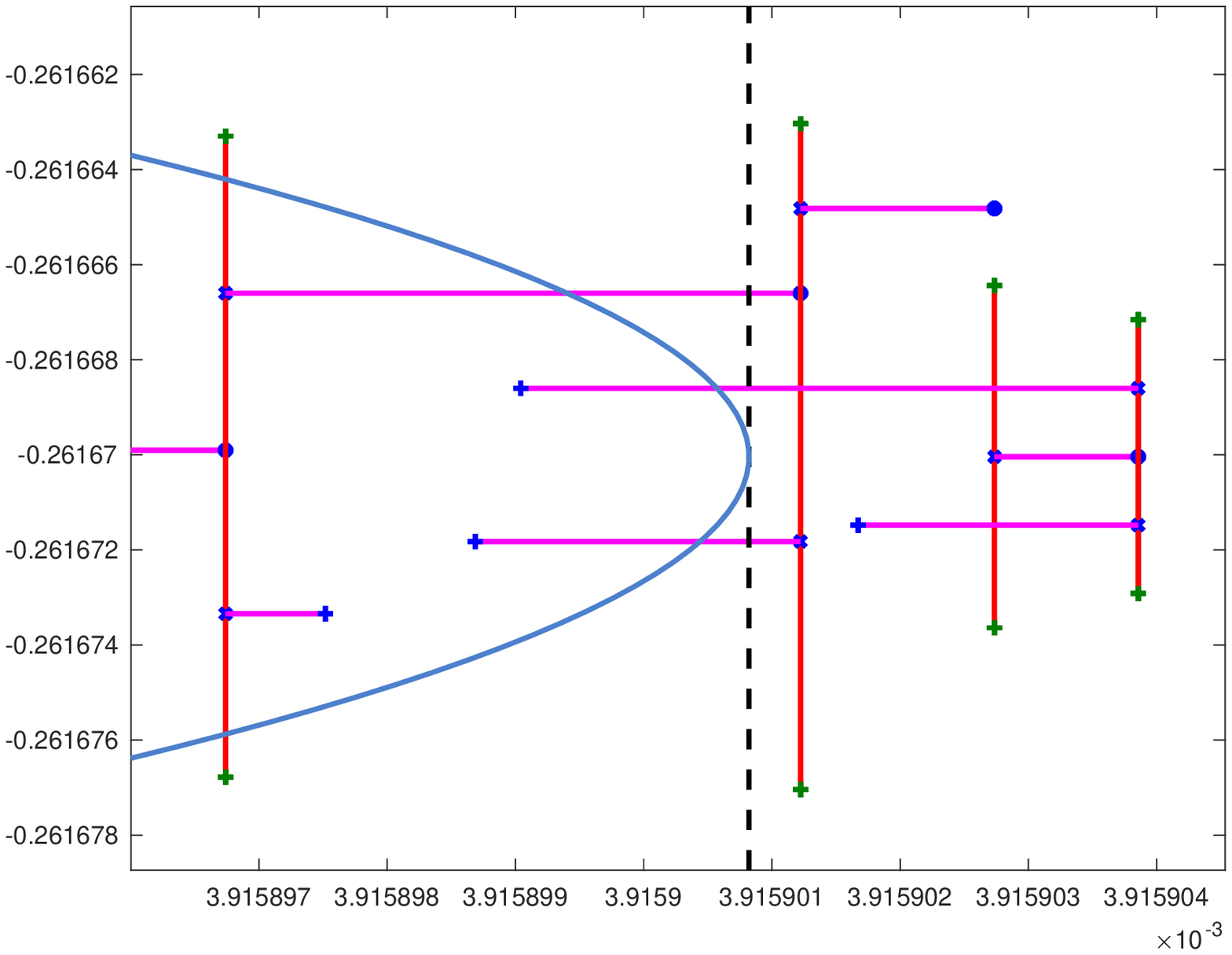} 
\label{fig:error_eig}
}
\subfloat[\texttt{pspa} with \texttt{SLICOT}]{
\includegraphics[scale=.2395,trim={0.8cm 1cm 1.86cm 1cm},clip]{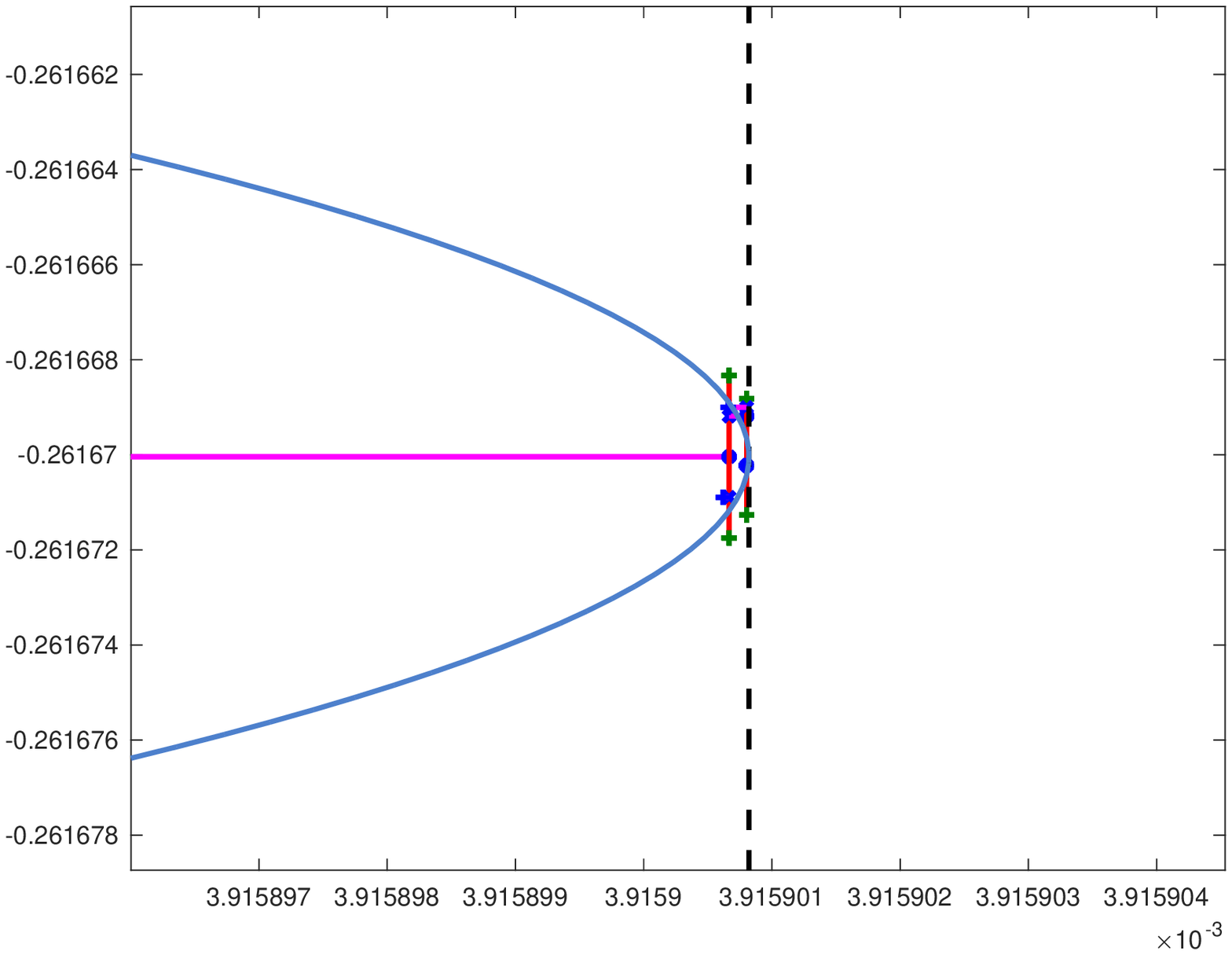} 
\label{fig:error_slicot}
}
\subfloat[\texttt{pspa} with \texttt{SLICOT} (close up)]{
\includegraphics[scale=.2395,trim={0.6cm 1cm 1.86cm 1cm},clip]{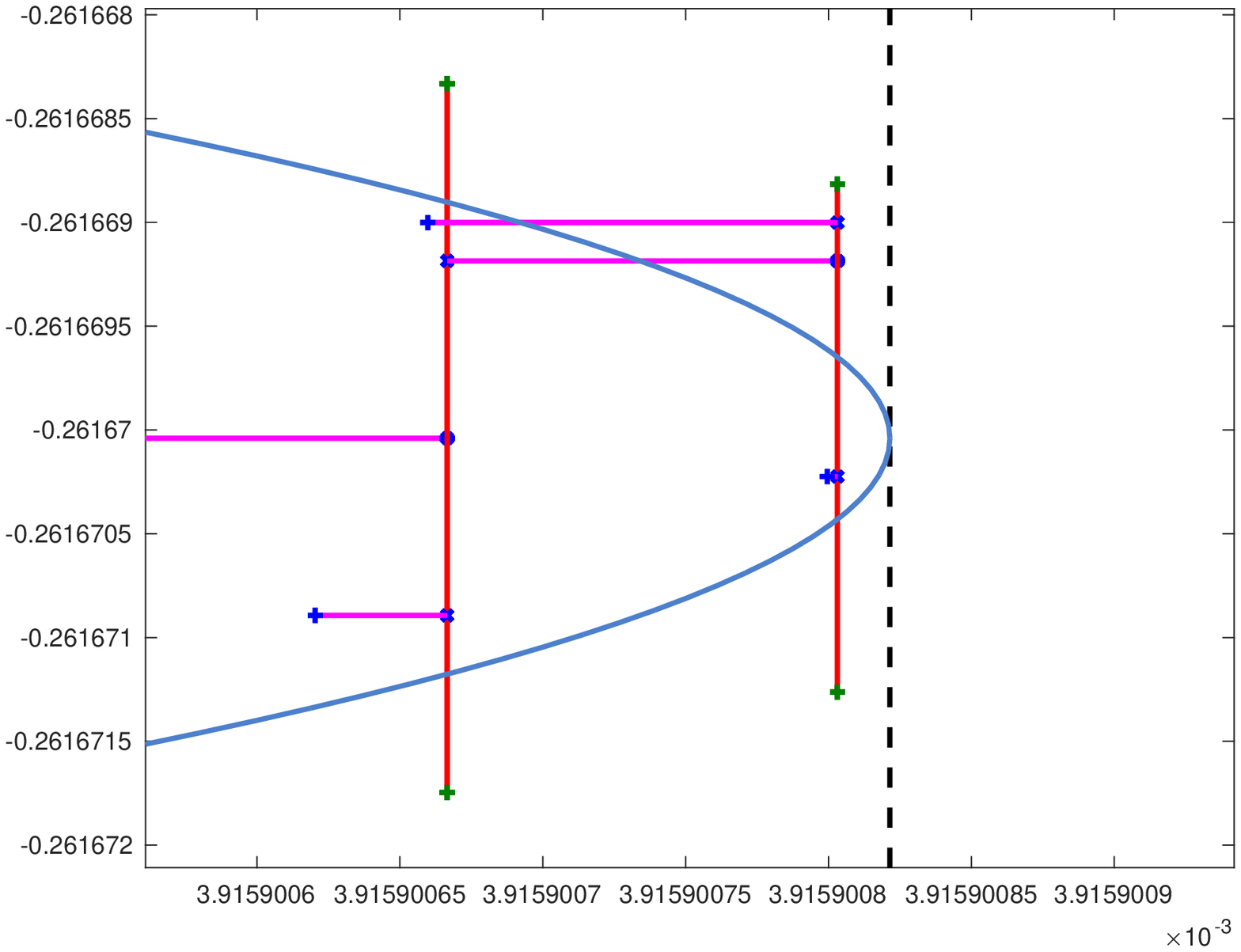} 
\label{fig:error_slicot_zoom}
}
\caption{The plotting output of \texttt{pspa} is shown for \texttt{orrsommerfeld\_demo(201)} and $\eps=10^{-4}$, 
demonstrating how inaccuracies in the eigensolves can lead to a significant loss of digits. 
The value of the pseudospectral abscissa computed by our improved method is shown by the grey dashed line, while the pseudospectral value boundary is shown by the blue curve, computed via \texttt{contour}.
}
\label{fig:eig_errors}
\end{figure}

\section{Conclusion}
\label{sec:wrapup}
By extending and improving upon the $\eps$-pseudospectral abscissa and radius algorithms of \cite{BurLO03} and \cite{MenO05},
we developed the first algorithms to compute, not just approximate,
the general $\eps$-spectral value set abscissa and radius to high accuracy.
Our experiments validate that our new root-finding-based approach 
is noticeably faster and more accurate than directly-extend approaches,
benefits that are also relevant for pseudospectra.
The new methods also use a novel new technique for handling 
singular pencils and/or problematic interior searches.

\section*{Acknowledgement}
The authors are grateful to the referees for many helpful comments to improve the manuscript
and to Emre Mengi and Michael L. Overton for discussions regarding the numerical subtleties 
of their criss-cross codes for the pseudospectral abscissa and radius.

{\small
\bibliographystyle{alpha}
\bibliography{csc,mor,software}  
}
\clearpage

\appendix

\begin{center}
\huge Supplementary Appendices
\end{center}

\section{Proofs of Theorems~\ref{thm:eigsing_cont}, \ref{thm:anyline_search}, and \ref{thm:eigsing_disc}}
\label{apdx:proofs}
\subsection{Proof of Theorem~\ref{thm:eigsing_cont}}
\label{apdx:eigsing_cont_proof}
\begin{proof}
Let $\gamma$ be a singular value of $G(\xiy)$ with left and right singular vectors $u$ and $v$,
that is, so that $G(\xiy)v = \gamma u$ and $G(\xiy)^*u = \gamma v$.  
Using the expanded versions of these two equivalences 
\beq
\begin{aligned}
	\label{eq:tfsv_equiv_cont}
	\left( \tfqs{(\xiy)} \right) v 	&= \gamma u 
	\quad \text{and} \\
	\left( \tfqs{(\xiy)} \right)^* u  &= \gamma v,
\end{aligned}
\eeq
we define 
\beq
	\label{eq:qs_cont}
	q = \left( (\xiy) E - A \right)^{-1}Bv 
	\quad \text{and} \quad 
	s = \left((\xiyconj) E^* - A^* \right)^{-1}C^*u.
\eeq
Rewriting \eqref{eq:tfsv_equiv_cont} using \eqref{eq:qs_cont} yields the following matrix equation:
\beq
	\label{eq:uv_cont}
	\begin{bmatrix} C & 0 \\ 0 & B^* \end{bmatrix} 
	\begin{bmatrix} q \\ s \end{bmatrix} 
	= 
	\begin{bmatrix} -D & \gamma I \\ \gamma I & -D^* \end{bmatrix} 
	\begin{bmatrix} v \\ u \end{bmatrix} 
	~\Longrightarrow~
	\begin{bmatrix} v \\ u \end{bmatrix} 
	= 
	\begin{bmatrix} -D & \gamma I \\ \gamma I & -D^* \end{bmatrix}^{-1}
	\begin{bmatrix} C & 0 \\ 0 & B^* \end{bmatrix} 
	\begin{bmatrix} q \\ s \end{bmatrix} ,
\eeq
where
\beq
	\label{eq:Dgamma_inv_cont}
	\begin{bmatrix} -D & \gamma I \\ \gamma I & -D^* \end{bmatrix}^{-1}
	= 
	\begin{bmatrix} -R^{-1}D^* & -\gamma R^{-1} \\ -\gamma S^{-1} & -DR^{-1} \end{bmatrix}
	\quad \text{and} \quad
	\begin{bmatrix} q \\ s \end{bmatrix} \ne 0.
\eeq
Rewriting \eqref{eq:qs_cont} as a matrix equation gives:
\beq
	\label{eq:EAqs1_cont}
	\left(
	\begin{bmatrix} (\xiy) E & 0 \\ 0 & (\xiyconj) E^* \end{bmatrix} -
	\begin{bmatrix} A & 0 \\ 0 & A^* \end{bmatrix}
	\right)
	\begin{bmatrix} q \\ s \end{bmatrix}
	= 
	\begin{bmatrix} B & 0 \\ 0 & C^* \end{bmatrix} 
	\begin{bmatrix} v \\ u \end{bmatrix}.
\eeq
Substituting in \eqref{eq:uv_cont} for the rightmost term of \eqref{eq:EAqs1_cont} yields
\beq
	\label{eq:EAqs2_cont}
	\begingroup
	\setlength\arraycolsep{1.7pt}
	\left(
	\begin{bmatrix} (\xiy)E & 0 \\ 0 & (\xiyconj) E^* \end{bmatrix}  -
	\begin{bmatrix} A & 0 \\ 0 & A^* \end{bmatrix} 
	\right)
	\begin{bmatrix} q \\ s \end{bmatrix} 
	= 
	\begin{bmatrix} B & 0 \\ 0 & C^* \end{bmatrix} 
	\begin{bmatrix} -D & \gamma I \\ \gamma I & -D^* \end{bmatrix}^{-1}
	\begin{bmatrix} C & 0 \\ 0 & B^* \end{bmatrix} 
	\begin{bmatrix} q \\ s \end{bmatrix} .
	\endgroup
\eeq
Bringing over terms from the left side to separate out $\iy$ and 
substituting the inverse on the right using \eqref{eq:Dgamma_inv_cont} and then multiplying out the matrix terms, 
we have
\beqs
	\label{eq:EAqs3_cont}
	\begingroup
	\setlength\arraycolsep{3pt}
	\iy \begin{bmatrix} E & 0 \\ 0 & -E^* \end{bmatrix}  
	\begin{bmatrix} q \\ s \end{bmatrix} 
	=
	\begin{bmatrix} A -xE & 0 \\ 0 & A^* -xE^* \end{bmatrix} 
	\begin{bmatrix} q \\ s \end{bmatrix} 
	+ 
	\begin{bmatrix} -BR^{-1}D^*C  & -\gamma BR^{-1}B^* \\
		 -\gamma C^*S^{-1}C & -C^*DR^{-1}B^* \end{bmatrix} 
	\begin{bmatrix} q \\ s \end{bmatrix} .
	\endgroup
\eeqs
Combining the matrices on the right and multiplying by
\[
	\begin{bmatrix} I & 0\\ 0 & -I \end{bmatrix}
\]
yields:
\beqs
	\begingroup
	\setlength\arraycolsep{3pt}
	\iy \begin{bmatrix} E & 0 \\ 0 & E^* \end{bmatrix}  
	\begin{bmatrix} q \\ s \end{bmatrix} 
	=
	\begin{bmatrix} (A -xE -BR^{-1}D^*C)  & -\gamma BR^{-1}B^* \\
		 \gamma C^*S^{-1}C & -(A -xE -BR^{-1}D^*C)^* \end{bmatrix} 
	\begin{bmatrix} q \\ s \end{bmatrix} .
	\endgroup
\eeqs
It is now clear that $\iy$ is an eigenvalue of the matrix pencil $\MNpencont$.

Now suppose that $\iy$ is an eigenvalue of pencil $\MNpencont$ with eigenvector
given by $q$ and $s$ as above.  Then it follows that \eqref{eq:EAqs2_cont} holds, which can
be rewritten as \eqref{eq:EAqs1_cont} by defining $u$ and $v$ using the right-hand side equation of \eqref{eq:uv_cont}, noting that neither can be identically zero.  It is then clear that the two equivalences in \eqref{eq:qs_cont} both hold.  Finally, substituting \eqref{eq:qs_cont} into the left-hand side equation of \eqref{eq:uv_cont}, it is clear that $\gamma$ is a singular value 
of $G(\xiy)$, with left and right singular vectors $u$ and $v$.
\end{proof}

\subsection{Proof of Theorem~\ref{thm:anyline_search}}
\begin{proof}
\sloppy
By Theorem~\ref{thm:eigsing_cont} and Corollary~\ref{cor:svs_tf}, 
$-s + \imagunit \omega_j$ must be all the boundary points of 
$\SVSeps( e^{\imagunit \theta_\mathrm{r}} A, e^{\imagunit \theta_\mathrm{r}} B, C, D, E)$ along the vertical line defined by $x=-s$.
Recalling \eqref{eq:MDelta}, 
since this spectral value set is entirely composed of eigenvalues of $(e^{\imagunit \theta_\mathrm{r}} M(\Delta),E)$, 
multiplying $e^{\imagunit \theta_\mathrm{r}} M(\Delta)$ by $e^{-\imagunit \theta_\mathrm{r}}$ is equivalent to a rotation 
about the origin by angle $-\theta_\mathrm{r}$, which yields $\SVSeps(A,B,C,D,E)$.
Since $\theta_\mathrm{r} = \pi/2 - \theta$, this specific rotation also moves all points $-s + \imagunit \omega_j$ precisely onto the line $L(\theta,s)$
and thus $\lambda_j = e^{-\imagunit \theta_\mathrm{r}}(-s + \imagunit \omega_j)$ are all the boundary points 
of $\SVSeps(A,B, C, D, E)$ along $L(\theta,s)$. 
\end{proof}

\subsection{Proof of Theorem~\ref{thm:eigsing_disc}}
\label{apdx:eigsing_disc_proof}
\begin{proof}
Let $\gamma$ be a singular value of $G(r\eitheta)$ with left and right singular vectors $u$ and $v$,
that is, so that $G(r\eitheta)v = \gamma u$ and $G(r\eitheta)^*u = \gamma v$.  
Using the expanded versions of these two equivalences 
\beq
	\label{eq:tfsv_equiv_disc}
	\left( \tfqs{r\eitheta} \right) v 	= \gamma u 
	\quad \text{and} \quad 
	\left( \tfqs{r\eitheta} \right)^* u  = \gamma v,
\eeq
we define 
\beq
	\label{eq:qs_disc}
	q = \left( r\eitheta E - A \right)^{-1}Bv 
	\quad \text{and} \quad 
	s = \left( r\eithetaconj E^* - A^* \right)^{-1}C^*u.
\eeq
Similar to the proof of Theorem~\ref{thm:eigsing_cont}, it follows that 
\beq
	\label{eq:EAqs2_disc}
	\setlength\arraycolsep{3pt}
	\begingroup
	\left(
	\begin{bmatrix} \eitheta rE & 0 \\ 0 & r\eithetaconj E^* \end{bmatrix}  -
	\begin{bmatrix} A & 0 \\ 0 & A^* \end{bmatrix} 
	\right)
	\begin{bmatrix} q \\ s \end{bmatrix} 
	= 
	\begin{bmatrix} B & 0 \\ 0 & C^* \end{bmatrix} 
	\begin{bmatrix} -D & \gamma I \\ \gamma I & -D^* \end{bmatrix}^{-1}
	\begin{bmatrix} C & 0 \\ 0 & B^* \end{bmatrix} 
	\begin{bmatrix} q \\ s \end{bmatrix} .
	\endgroup
\eeq
Furthermore, the rightmost three terms of \eqref{eq:EAqs2_disc} can again be replaced by first substituting 
the matrix inverse with its explicit form given by \eqref{eq:Dgamma_inv_cont} and then multiplying these three terms together.  
Then, multiplying on the left by 
\[
	\begin{bmatrix} I & 0 \\ 0 & -\eitheta I \end{bmatrix}
\]
and rearranging terms yields
\beqs
	\eitheta \begin{bmatrix} rE & 0 \\ 0 & A^* \end{bmatrix}  
	\begin{bmatrix} q \\ s \end{bmatrix} 
	=
	\begin{bmatrix} A & 0 \\ 0 & rE^* \end{bmatrix} 
	\begin{bmatrix} q \\ s \end{bmatrix} 
	+ 
	\begin{bmatrix} B & 0 \\ 0 & -\eitheta C^* \end{bmatrix}  
	\begin{bmatrix} -R^{-1}D^*C  & -\gamma R^{-1}B^* \\
		 -\gamma S^{-1}C & -DR^{-1}B^* \end{bmatrix} 
	\begin{bmatrix} q \\ s \end{bmatrix} .
\eeqs
Separating and then bringing the $-\eitheta$ terms over to the left side, we obtain 
\beqs
	\eitheta \begin{bmatrix} rE & 0 \\ -\gamma C^*S^{-1}C & A^* - C^*DR^{-1}B^* \end{bmatrix}  
	\begin{bmatrix} q \\ s \end{bmatrix} 
	=
	\begin{bmatrix} A -  BDR^{-1}B^* & -\gamma BR^{-1}B^* \\ 0 & rE^* \end{bmatrix}
	\begin{bmatrix} q \\ s \end{bmatrix} ,
\eeqs
and thus it is clear that $\eitheta$ is an eigenvalue of the matrix pencil $\MNpendisc$.

The reverse implication follows similarly to the reverse argument given in Appendix~\ref{apdx:eigsing_cont_proof} for the proof of Theorem~\ref{thm:eigsing_cont}.
\end{proof}

\section{Pseudospectral evaluation (complete version)}
\label{apdx:pseudo}
To compare our new improved algorithms with the original criss-cross methods for 
computing the pseudospectral abscissa and radius,
we also tested against \texttt{pspa} and \texttt{pspr}, respectively.
We used 20 of the 21 examples from EigTool, all of order 200 (i.e. $n=m=p=200$),
using $\eps=0.01$;
we could not include \texttt{companion\_demo} as it contains
\texttt{infs} and \texttt{nans} at this size.
To collect the same detailed performance data as provided in \S\ref{sec:num_svs}, 
we added simple counters inside the \texttt{pspa} and \texttt{pspr} routines.
Table~\ref{table:ps_abs} and \ref{table:ps_rad} give the respective performance comparisons for the 
pseudospectral abscissa and radius cases.

Note that when $B=C=I_n$ and $D=0$, by default \texttt{specValSet} 
does not compute the largest singular value of $G(\lambda)$ but instead equivalently
computes the reciprocal of the smallest singular value of $\lambda E - A$.
This is significantly more efficient  than using $G(\lambda)$, which requires $(\lambda E - A)^{-1}$. 
Furthermore, with this smallest singular value approach,
the cost of obtaining the first and second derivatives is essentially negligible.

However, during testing and development of \texttt{specValSet}, we noticed that \texttt{svd} may sometimes return extremely inaccurate results for the smaller singular values when singular vectors are also requested 
(which we need for the first and second derivatives).
This numerical issue appears to be due to the underlying \texttt{GESDD} routine from LAPACK, which 
is used in MATLAB whenever singular vectors are requested 
(right \emph{and} left on R2017b and earlier and right \emph{or}
 left on R2018a and newer) and the minimum dimension 
of the matrix is at least 26.
If a given matrix $A$ is very poorly scaled, 
\texttt{GESDD} tends to compute all singular values below $\|A\|_2 \cdot \emach$
as all being about $\|A\|_2 \cdot \emach$.
This means these ``computed" smaller singular values
may have \emph{zero} digits of accuracy.
As this can be quite problematic,  
\texttt{specValSet} also allows the user to optionally revert to the more expensive choice of 
computing the largest singular value of $G(\lambda)$, as a temporary workaround
until \texttt{svd} and \texttt{GESDD} are improved.
We have notified the LAPACK maintainers and MathWorks about this issue with \texttt{GESDD} and \texttt{svd}.\footnote{For more info, see \url{https://github.com/Reference-LAPACK/lapack/issues/316}.}

So far, we have only observed this bad numerical behavior of \texttt{svd} occurring on one example,
\texttt{companion\_demo}, which is not in our test set anyway due to its extreme scaling
(the norm of \texttt{companion\_demo(26)} is $6.09 \times 10^{26}$ and this grows as $n$ is increased).
As such, we still used the more efficient smallest singular value of $\lambda E -A$ approach for 
all problems in the evaluation.
Furthermore, for all, we also confirmed that there were no 
accuracy issues with the pseudospectral abscissa and radius values 
computed by our new methods.
Nevertheless, it is conceivable that this numerical issue with \texttt{svd} may have resulted in less accurately computed  first- and second- order derivatives,
which in turn, could have caused more function evaluations in the rooting finding than should have been necessary.

\begin{table}[t]
\footnotesize
\setlength{\tabcolsep}{3.5pt}
\robustify\bfseries
\center
\begin{tabular}{ l | cc | cc | cc | cl | @{}S@{} S | c } 
\toprule
\multicolumn{12}{c}{Pseudospectral Abscissa ($\eps=0.01$): \texttt{pspa} versus new method} \\
\midrule
\multicolumn{1}{c}{} & 
	\multicolumn{4}{c}{\# solves} & \multicolumn{4}{c}{\# searches} & 
	\multicolumn{3}{c}{} \\
\cmidrule(lr){2-5}
\cmidrule(lr){6-9}
\multicolumn{1}{l}{Problem} & 
	\multicolumn{2}{c}{Eig} & \multicolumn{2}{c}{SVD} & 
	\multicolumn{2}{c}{vert.} & \multicolumn{2}{c}{horz.} & 	
	\multicolumn{2}{c}{time (sec.)} & \multicolumn{1}{c}{\% faster}\\
\midrule
\texttt{airy(201)}               &  13 &    4 &   10 &   38 &   4 &   4 &   9 &     4(6) &            1.748 &  \bfseries 0.886 &    97 \\
\texttt{basor(200)}              &   9 &    2 &    0 &   13 &   4 &   2 &   5 &     2(2) &            1.446 &  \bfseries 0.494 &   193 \\
\texttt{chebspec(201)}           &   5 &    2 &    4 &   34 &   2 &   2 &   3 &     2(3) &            0.626 &  \bfseries 0.485 &    29 \\
\texttt{convdiff(201)}           &   4 &    1 &    0 &   23 &   2 &   1 &   2 &     1(1) &            0.349 &  \bfseries 0.188 &    85 \\
\texttt{davies(201)}             &   5 &    1 &    0 &    6 &   2 &   1 &   3 &     1(1) &            0.787 &  \bfseries 0.229 &   243 \\
\texttt{demmel(200)}             &  15 &    6 &   12 &   75 &   7 &   6 &   8 &     6(6) &            1.692 &  \bfseries 1.222 &    39 \\
\texttt{frank(200)}              &   3 &    1 &    0 &   14 &   2 &   1 &   1 &     1(1) &            0.414 &  \bfseries 0.181 &   129 \\
\texttt{gaussseidel({200,'C'})}  &   5 &    1 &   10 &    5 &   3 &   1 &   2 &     1(1) &            1.013 &  \bfseries 0.148 &   583 \\
\texttt{gaussseidel({200,'D'})}  &   5 &    2 &    0 &   20 &   3 &   2 &   2 &     2(3) &            0.754 &  \bfseries 0.375 &   101 \\
\texttt{gaussseidel({200,'U'})}  &   5 &    2 &    0 &   20 &   3 &   2 &   2 &     2(3) &            1.044 &  \bfseries 0.536 &    95 \\
\texttt{grcar(200)}              &   3 &    1 &    4 &   11 &   2 &   1 &   1 &     1(2) &            0.418 &  \bfseries 0.205 &   104 \\
\texttt{hatano(200)}             &   4 &    1 &    0 &    5 &   2 &   1 &   2 &     1(1) &            0.598 &  \bfseries 0.121 &   395 \\
\texttt{kahan(200)}              &   4 &    1 &    0 &    8 &   2 &   1 &   2 &     1(1) &            0.478 &  \bfseries 0.138 &   247 \\
\texttt{landau(200)}             &  14 &    2 &    2 &   13 &   5 &   2 &   9 &     2(2) &            1.091 &  \bfseries 0.334 &   227 \\
\texttt{orrsommerfeld(201)}$^*$  &  15 &    4 &    2 &   34 &   7 &   4 &   8 &     4(4) &            1.900 &  \bfseries 0.820 &   132 \\
\texttt{random(200)}             &   4 &    2 &    0 &   16 &   2 &   2 &   2 &     2(3) &            0.557 &  \bfseries 0.334 &    67 \\
\texttt{randomtri(200)}          &   3 &    1 &   40 &   10 &   2 &   1 &   1 &     1(1) &            0.639 &  \bfseries 0.152 &   322 \\
\texttt{riffle(200)}             &   7 &    1 &    2 &    8 &   4 &   1 &   3 &     1(1) &            0.427 &  \bfseries 0.099 &   331 \\
\texttt{transient(200)}          &   6 &    2 &    0 &   16 &   3 &   2 &   3 &     2(2) &            1.180 &  \bfseries 0.618 &    91 \\
\texttt{twisted(200)}            &   9 &    2 &    8 &   14 &   4 &   2 &   5 &     2(2) &            1.349 &  \bfseries 0.339 &   298 \\
\midrule
\multicolumn{1}{r|}{Totals:} & 138 & 39 & 94 & 383 & \multicolumn{6}{r|}{Average \% faster:} & 190\\ 
\bottomrule
\end{tabular}
\caption{
The headers remain as described in Table~\ref{table:svs_abs} in the main paper,
except now performance data is given for
\texttt{pspa} (left columns) and our improved approach (right columns)
for computing the pseudospectral abscissa on problems from EigTool, all of order 200 ($=n=m=p$).
}
\label{table:ps_abs}
\end{table}

Returning to our performance comparison, as mentioned in \S\ref{sec:num_ps}, our new method was on average 190\% and 84\% faster for the pseudospectral abscissa and radius cases, respectively.
In contrast to our spectral value set evaluation, where the DE variants were significantly less accurate
than our newer methods on four of the problems, \texttt{pspa} and \texttt{pspr} returned answers which had a high numerical agreement with the accurate values computed by our improved methods on all but one problem: \texttt{orrsommerfeld\_demo} for the abscissa case, where the relative error was $1.75 \times 10^{-9}$.
Rounding errors in the eigenvalue value computations caused the 
horizontal searches to repeatedly overshoot the true abscissa value;
\texttt{pspa} not only incurred more iterations than necessary, it did so while also making its accuracy even worse. 
In Figure~\ref{fig:eig_errors}, an example of this phenomenon is shown 
for computing the pseudospectral abscissa of \texttt{orrsommerfeld\_demo(201)}
with $\eps=10^{-4}$, where the relative error was even more pronounced: $7.75 \times 10^{-7}$.
When replacing \texttt{eig} by a structure-preserving eigensolver from SLICOT,
the relative error from \texttt{pspa} was $-4.68\times10^{-9}$; in this case,
\texttt{pspa} stagnated a bit too early, due to the vertical search failing to return 
the vertical cross section.

\begin{table}[t]
\footnotesize
\setlength{\tabcolsep}{3.5pt}
\robustify\bfseries
\center
\begin{tabular}{ l | cc | cc | cc | cl | @{}S@{} S | c} 
\toprule
\multicolumn{12}{c}{Pseudospectral Radius ($\eps=0.01$): \texttt{pspr} versus new method} \\
\midrule
\multicolumn{1}{c}{} & 
	\multicolumn{4}{c}{\# solves} & \multicolumn{4}{c}{\# searches} & 
	\multicolumn{3}{c}{} \\
\cmidrule(lr){2-5}
\cmidrule(lr){6-9}
\multicolumn{1}{l}{Problem} & 
	\multicolumn{2}{c}{Eig} & \multicolumn{2}{c}{SVD} & 
	\multicolumn{2}{c}{circ.} & \multicolumn{2}{c}{rad.} & 
	\multicolumn{2}{c}{time (sec.)} & \multicolumn{1}{c}{\% faster}\\
\midrule
\texttt{airy(201)}               &   5 &    2 &   19 &   28 &   2 &   2 &   3 &     2(3) &            0.953 &  \bfseries 0.800 &    19 \\
\texttt{basor(200)}              &   7 &    2 &   15 &   18 &   3 &   2 &   4 &     2(2) &            1.531 &  \bfseries 0.868 &    76 \\
\texttt{chebspec(201)}           &   2 &    1 &   11 &   14 &   1 &   1 &   1 &     1(1) &            0.315 &  \bfseries 0.292 &     8 \\
\texttt{convdiff(201)}           &   2 &    1 &    9 &   11 &   1 &   1 &   1 &     1(1) &            0.289 &  \bfseries 0.225 &    28 \\
\texttt{davies(201)}             &   2 &    1 &   10 &   13 &   1 &   1 &   1 &     1(1) &  \bfseries 0.409 &            0.424 &    -4 \\
\texttt{demmel(200)}             &   2 &    1 &    7 &   16 &   1 &   1 &   1 &     1(1) &            0.275 &  \bfseries 0.234 &    17 \\
\texttt{frank(200)}              &   3 &    1 &    6 &   19 &   1 &   1 &   2 &     1(1) &            0.494 &  \bfseries 0.290 &    71 \\
\texttt{gaussseidel({200,'C'})}  &   2 &    1 &    2 &   10 &   1 &   1 &   1 &     1(1) &            0.403 &  \bfseries 0.299 &    35 \\
\texttt{gaussseidel({200,'D'})}  &   6 &    1 &   10 &   17 &   3 &   1 &   3 &     1(2) &            0.993 &  \bfseries 0.312 &   218 \\
\texttt{gaussseidel({200,'U'})}  &   4 &    1 &   11 &   21 &   2 &   1 &   2 &     1(3) &            1.245 &  \bfseries 0.415 &   200 \\
\texttt{grcar(200)}              &  10 &    5 &   34 &   57 &   5 &   5 &   5 &     6(6) &            1.969 &  \bfseries 1.597 &    23 \\
\texttt{hatano(200)}             &   2 &    1 &    6 &   10 &   1 &   1 &   1 &     1(1) &            0.402 &  \bfseries 0.294 &    37 \\
\texttt{kahan(200)}              &   7 &    1 &   19 &   12 &   3 &   1 &   4 &     1(2) &            1.103 &  \bfseries 0.241 &   357 \\
\texttt{landau(200)}             &   7 &    2 &   16 &   16 &   3 &   2 &   4 &     2(2) &            1.419 &  \bfseries 0.547 &   160 \\
\texttt{orrsommerfeld(201)}      &   2 &    1 &   11 &   14 &   1 &   1 &   1 &     1(1) &            0.666 &  \bfseries 0.424 &    57 \\
\texttt{random(200)}             &   6 &    4 &   26 &   48 &   3 &   4 &   3 &     6(6) &            1.546 &  \bfseries 1.368 &    13 \\
\texttt{randomtri(200)}          &   5 &    1 &   15 &   14 &   2 &   1 &   3 &     1(2) &            0.891 &  \bfseries 0.252 &   254 \\
\texttt{riffle(200)}             &   2 &    1 &    2 &   14 &   1 &   1 &   1 &     1(2) &            0.292 &  \bfseries 0.227 &    28 \\
\texttt{transient(200)}          &   6 &    2 &   11 &   21 &   2 &   2 &   4 &     2(2) &            1.755 &  \bfseries 1.009 &    74 \\
\texttt{twisted(200)}            &   2 &    1 &    2 &   17 &   1 &   1 &   1 &     1(1) &            0.412 &  \bfseries 0.411 &     0 \\
\midrule
\multicolumn{1}{r|}{Totals:} & 84 & 31 & 242 & 390 & \multicolumn{6}{r|}{Average \% faster:} & 84\\ 
\bottomrule
\end{tabular}
\caption{
The headers remain mostly unchanged from Table~\ref{table:ps_abs},
except now \texttt{pspr} (left columns) and our improved approach (right columns)
for computing the pseudospectral radius for the same problems;
correspondingly, the number of circular and radial searches are respectively given under the ``circ." and ``rad." headers.
}
\label{table:ps_rad}
\end{table}

In Table~\ref{table:ps_abs}, we see that our new method was faster than \texttt{pspa} on every single test problem and 190\% faster on average.  The largest performance gap was on \texttt{gaussseidel(200,'C')},
where our method was 583\% faster than \texttt{pspa}. 
The smallest performance gap was on \texttt{chebspec(201)}, where our method was 29\% faster than \texttt{pspa}.  
Over all problems, relative to \texttt{pspa}, we see that our method only needed about a quarter of the expensive eigenvalue computations, but required slightly more than four times the number of SVDs.
Nevertheless, the tradeoff was a success given the clear overall large reductions in running times.

For the pseudospectral radius comparison, shown in Table~\ref{table:ps_rad}, 
we see that our new method was faster on 19 out of the 20 problems,
but to a lesser degree.  At best, our new method was 357\% faster than \texttt{pspr} (on \texttt{kahan(200)})
and 84\% faster on average.  On the only problem where our new method was slower than \texttt{pspr} (\texttt{davies(201)}), 
the performance difference was rather negligible at just 4\% slower.
The smaller performance improvement on the radius problems
appears to be due to the fact that on ten of the problems, \texttt{pspr} only needed just one circular search
before convergence was met.
Indeed, compared to the abscissa case, 
the total number of expensive eigenvalue computations incurred by \texttt{pspr} was
simply much less than that by \texttt{pspa}, as was the relative reduction of them afforded by our new method.

\end{document}